\documentclass[]{article}
\usepackage{songmei}
\usepackage[toc,page]{appendix}
\usepackage{subfigure}
\usepackage{csquotes}
\usepackage{mathtools}
\usepackage[utf8]{inputenc} 
\mathtoolsset{showonlyrefs}

\def\op{{\rm op}}

\def\bR{{\boldsymbol R}}
\def\Id{{\mathbf I}}

\def\bX{{\boldsymbol X}}

\def\cL{{\mathcal L}}
\def\sb{{\mathsf b}}
\def\sx{{\mathsf x}}
\def\bGamma{\boldsymbol{\Gamma}}

\def\de{{\rm d}}
\def\bx{{\boldsymbol x}}
\def\by{{\boldsymbol y}}
\def\bW{{\boldsymbol W}}
\def\ba{{\boldsymbol a}}
\def\cF{{\mathcal F}}
\def\Unif{{\sf Unif}}
\def\normal{{\sf N}}

\def\bv{{\boldsymbol V}}
\def\bM{{\boldsymbol M}}

\def\bT{\boldsymbol{T}}
\def\bS{\boldsymbol{S}}

\def\bz{{\boldsymbol z}}
\def\proj{{\mathsf P}}

\def\bB{{\boldsymbol B}}

\def\be{{\boldsymbol e}}
\def\bu{{\boldsymbol u}}
\def\bg{{\boldsymbol g}}

\def\bA{{\boldsymbol A}}

\def\bv{{\boldsymbol v}}
\def\bxi{{\boldsymbol \xi}}

\def\reals{{\mathbb R}}

\def\de{{\rm d}}
\def\bx{{\boldsymbol x}}
\def\by{{\boldsymbol y}}
\def\bW{{\boldsymbol W}}
\def\bG{{\boldsymbol G}}
\def\ba{{\boldsymbol a}}
\def\cF{{\mathcal F}}
\def\Unif{{\rm Unif}}

\def\bv{{\boldsymbol v}}
\def\bz{{\boldsymbol z}}
\def\proj{{\mathsf P}}

\def\normal{{\sf N}}

\def\be{{\boldsymbol e}}
\def\bu{{\boldsymbol u}}
\def\bg{{\boldsymbol g}}

\def\bA{{\boldsymbol A}}

\def\cS{{\mathcal S}}

\def\bY{{\boldsymbol Y}}

\def\GOE{{\rm GOE}}

\def\bbm{{\boldsymbol m}}
\def\sh{{\mathsf h}}

\def\cS{{\mathcal S}}

\def\bWtilde{\widetilde{\boldsymbol{W}}{}}
\def\bK{\boldsymbol{K}}

\newcommand{\indic}[1]{\mathbf{1}_{#1}}

\def\muhat{\widehat{\mu}{}}

\def\sB{\mathsf{B}}

\def\pconv{\stackrel{\mathrm{p}}\rightarrow}
\def\Wconv{\stackrel{W_2}\rightarrow}

\def\AMP{\mathsf{AMP}}
\def\SF{\mathsf{SF}}

\def\bmhat{\widehat{\boldsymbol{m}}{}}

\def\bWtilde{\widetilde{\boldsymbol{W}}{}}

\def\sF{\mathsf{F}}

\def\brevebm{\breve{\boldsymbol{m}}{}}
\def\brevebM{\breve{\boldsymbol{M}}{}}
\def\brevebg{\breve{\boldsymbol{g}}{}}
\def\brevebG{\breve{\boldsymbol{G}}{}}

\def\sL{\mathsf{L}}

\def\sbhat{\widehat{\mathsf{b}}{}}

\DeclareMathOperator*{\plim}{p-lim}
\DeclareMathOperator*{\plimsup}{p-lim\,sup}
\DeclareMathOperator*{\pliminf}{p-lim\,inf}

\def\alg{\mathrm{alg}}
\def\dist{\mathsf{dist}}

\def\Law{\mathsf{Law}}
\def\FMM{\mathsf{FMM}}
\def\AMS{\mathsf{AMS}}
\def\SE{\mathsf{SE}}

\def\iid{\stackrel{\mathrm{iid}}\sim}

\usepackage{hyperref}
\hypersetup{
    colorlinks,
    linkcolor={blue!80!black},
    citecolor={green!50!black},
}
\colorlet{linkequation}{blue}

\begin{document}

\title{Sudakov-Fernique post-AMP,\\ and a new proof of the local convexity of the TAP free energy}
\author{Michael Celentano\thanks{Department of Statistics, University of California, Berkeley.
E-mail: \texttt{mcelentano@berkeley.edu}}}
\maketitle

\begin{abstract}
    In many problems in modern statistics and machine learning,
    it is often of interest to establish that a first order method on a non-convex risk function eventually enters a region of parameter space in which the risk is locally convex.
    We derive an asymptotic comparison inequality, which we call the \emph{Sudakov-Fernique post-AMP inequality}, which, in a certain class of problems involving a $\GOE$ matrix,
    is able to probe properties of an optimization landscape locally around the iterates of an approximate message passing (AMP) algorithm.
    As an example of its use,
    we provide a new, and arguably simpler, proof of some of the results of \cite{celentanoFanMei2021},
    which establishes that the so-called TAP free energy in the $\Z_2$-synchronization problem is locally convex in the region to which AMP converges.
    We further prove a conjecture of \cite{AlaouiMontanariSelke2022} involving the local convexity of a related but distinct TAP free energy,
    which, as a consequence, confirms that their algorithm efficiently samples from the Sherrington-Kirkpatrick Gibbs measure throughout the ``easy'' regime.
\end{abstract}

\section{Introduction}

Many problems in modern statistics and machine-learning involve optimizing a non-convex loss function.
In many instances, 
first-order methods like gradient descent, stochastic gradient descent, 
approximate message passing, and related algorithms have been observed, both empirically and theoretically, to perform well even when the objective is highly non-convex.
Examples include phase retrieval \cite{candesLi2015,chenCandes2015,Sanghavi2017,Chen2019,caiLi2016}, 
matrix completion,
blind deconvolution \cite{Ma2020},
neural networks \cite{veiga2022,goldt2019,sarao2020}, spiked matrix and tensor estimation \cite{Sarao_Mannelli_2020,manelliKrzakala2019,manelliBiroli2019,manelliBiroli2020}, Gaussian mixture classification \cite{mignacco2020}, among others.
In all of these examples, there are regimes in which first-order methods successfully find points with near-minimal objective value or learn parameters with good statistical performance.
Substantial research effort has been dedicated--- and continues to be dedicated ---to understanding what drives the success or failure of these methods in non-convex settings. 

In some instances, the success of first-order methods can be explained by the following two-stage phenomenon \cite{Chen2019,celentanoFanMei2021}.
Upon initialization,
the first-order method is in a region of parameter space where the landscape is non-convex.
In the first stage,
the algorithm makes progress towards the global optimum or builds correlation with the unknown signal,
despite the non-convexity of the surrounding landscape.
Often, 
the analysis in this stage is primarily statistical: 
in each step, the randomness in the data generating distribution is used to argue that, with high-probability, the algorithm continues to make progress towards the global optimum or the hidden signal. 
This statistical analysis typically fails after sufficiently many iterations because--- without resampling ---the iterates become correlated with the randomness in the data in a complex way which is difficult to control. 
It is important, then, that the algorithm eventually enters a second stage, in which it arrives in a region of parameter space which is locally convex.
Now,
the analysis becomes geometric:
the algorithm makes progress towards a local optimum based on standard results from convex optimization theory.
Because the statistical analysis in the first stage fails after sufficiently many iterations,
it is important that the argument in the second stage is geometric rather than statistical.
Once the local geometry of the landscape is established,
the analysis becomes worst-case and deterministic. 

In this paper, 
we present a comparison inequality--- \emph{the Sudakov-Fernique post-AMP inequality} ---which in a certain class of problems is able to probe properties of an optimization landscape locally around the iterates of an approximate message passing algorithm.
Thus, it is well suited to establishing the two-stage phenomenon described in the previous paragraph when AMP is used in the first stage.
Recent papers \cite{celentanoMontanari2020,montanariWu2022,celentanoCheng2021} have shown that a large class of first-order methods,
including gradient descent, accelerated gradient descent, and related algorithms on appropriate optimization objectives,
are equivalent to AMP algorithms after a change of variables.
Thus, the Sudakov-Fernique post-AMP inequality can also probe the local landscape surrounding the iterates of such methods.

As an example, 
we consider in detail the problem of $\Z_2$-synchronization and the corresponding optimization problem based on the TAP free energy.
This model and optimization problem are defined in Section \ref{sec:Z2-intro}. 
We demonstrate how to use the Sudakov-Fernique post-AMP inequality to probe the landscape of the TAP free energy locally around iterates of an AMP algorithm tailored to this model.
This provides a new--- and arguably simpler ---proof of one of the main results in the paper \cite{celentanoFanMei2021},
namely, that AMP eventually enters a region of strong convexity around a stationary point of the objective. 
More recently, 
the paper \cite{AlaouiMontanariSelke2022} considered the problem of efficiently sampling from the Sherrington-Kirkpatrick Hamiltonian,
which we also define in Section \ref{sec:Z2-intro}.
One part of their proof requires considering a different but closely related version of the $\Z_2$-synchronization problem and TAP free energy.
In order to establish the success of their algorithm,
they require establishing the convergence of AMP to a region of this TAP free energy which is strongly convex. 
Our proof of the local convexity of the TAP free energy considered by \cite{celentanoFanMei2021} is easily modified to arrive at the same result for the TAP free energy considered by \cite{AlaouiMontanariSelke2022}.
This confirms that the algorithm of \cite{AlaouiMontanariSelke2022}
successfully samples from the Sherrington-Kirkpatrick Hamiltonian throughout the computationally easy regime,
resolving the conjecture of computational tractability.

The remainder of the paper is organized as follows. 
In Section \ref{sec:Z2-intro}, 
we provide the necessary background on the $\Z_2$-synchronization problem and its corresponding approximate message passing algorithm and TAP free energy.
In Section \ref{sec:overview},
we provide an overview of our results.
In Section \ref{sec:lit},
we briefly review related literature.
The Sudakov-Fernique inequality is presented in Section \ref{sec:SF-post-AMP},
and our results on the local convexity of TAP free energy in $\Z_2$-synchronization are stated in Section \ref{sec:main}.
Finally, our proof of the local convexity of the TAP free energy is presented in Section \ref{sec:proofs}, which includes the application of the Sudakov-Fernique post-AMP inequality.
Some technical details are deferred to appendices.

\subsection{Background on $\Z_2$-synchronization and the Sherrington-Kirkpatrick model}
\label{sec:Z2-intro}

In the $\Z_2$-synchronization problem,
we observe the matrix 
\begin{equation}
\label{eq:obs}
    \bY = \frac{\lambda}{n} \bx \bx^\top + \bW  \in \reals^n,
\end{equation}
where $\bW \sim \GOE(n)$, and 
$x_i \iid \mathsf{Unif}\{-1,+1\}$ are unobserved and $\lambda$ is known.
The goal is to estimate the unobserved vector $\bx$ with respect to $\ell_2$-loss.
The observations Eq.~\eqref{eq:obs} only identify $\bx$ up to sign:
the posterior $\P(\bx \in \cdot \mid \bY)$ is 
equal to the the posterior $\P(-\bx \in \cdot \mid \bY)$.
Thus, we seek an estimator $\bmhat = \bmhat(\bY)$ for which $\E[\min\{\| \bmhat \pm \bx \|_2^2\}]/n$ is small.

Several aspects of this problem are now well-understood from both an information-theoretic and computational point of view.
The paper \cite{deshpande2016} studied the asymptotic Bayes risk in this model,
and established that for $\lambda < 1$, 
it is impossible to do better asymptotically (as $n \rightarrow \infty$) than the trivial estimator $\bmhat(\bY) = \bzero = \E[\bx]$.
For $\lambda > 1$, 
weak recovery becomes information-theoretically possible:
the Bayes estimator $\bmhat_{\mathrm{Bayes}}$ achieves a non-trivial correlation with the signal,
and $\E[\min\{\| \bmhat_{\mathrm{Bayes}} \pm \bx \|_2^2\}]/n$ converges to a constant in $(0,1)$ which \cite{deshpande2016} characterize via the solution to a non-linear fixed-point equation.
Moreover, for all $\lambda > 1$, \cite{deshpande2016} proposed a practical algorithm--- a so-called Approximate Message Passing (AMP) algorithm ---which they proved achieves the Bayes risk up to asymptotically negligible terms provided it was initialized appropriately.
Although they conjectured that their results would hold when using the leading eigenvector of $\bY$ as an initialization,
for technical reasons they only considered initializations based on independent side information.
For example,
their results handle initializations of the form
\begin{equation}
\label{eq:init}
    \by = \gamma_0 \bx + \sqrt{\gamma_0} \bg,
\end{equation}
for some $\gamma_0 > 0$ arbitrarily small but non-vanishing asymptotically, 
where $\bg \sim \normal(0,\id_n)$.
Explicitly, their algorithm takes the form
\begin{equation}\label{eq:AMP-FMM}
\begin{aligned}
    \bz^0 &= \by,
    \qquad\qquad\qquad&
    \bbm^s &= \tanh(\bz^s),
    \\
    \bz^{s+1} &= \lambda \bY \bbm^s - \lambda \sbhat_s\bbm^{s-1},
    \qquad\qquad\qquad
    &
    \sbhat_s &= \lambda(1 - Q(\bbm^s)),
\end{aligned}
\end{equation}
where $\bbm \in [-1,1]^n$, $Q(\bbm) := \| \bbm \|_2^2 / n$,
and $\tanh(\,\cdot\,)$ is applied coordinate-wise.
They show that 
\begin{equation}
    \lim_{s \rightarrow \infty} \lim_{n \rightarrow \infty} 
        \frac1n\E[\| \bmhat^s - \bx \|_2^2]
        =
        \lim_{n \rightarrow \infty} 
        \frac1n\E[\min\{\| \bmhat_{\mathrm{Bayes}} \pm \bx \|_2^2\}],
\end{equation}
where the limits in $n$ are taken almost-surely.
Note that the initialization in Eq.~\eqref{eq:init} breaks the symmetry, so that no minimum is needed inside the expectation on the left-hand side.
Nevertheless, it is unsatisfying for the result to rely on small but non-negligible side information,
which may not be available.
The later paper \cite{montanari2021} establishes the same result as \cite{deshpande2016} when a spectral initialization is used;
that is, $\bz^0$ is taken to be the leading eigenvector of $\bY$, appropriately scaled.
For simplicity,
in this paper we will study the algorithm of \cite{deshpande2016} with initialization \eqref{eq:init},
although our results are easily extended to the case in which a spectral initialization is used (see Remark \ref{rmk:spectral}).

An alternative but related approach is based on optimization.
This approach, 
studied by \cite{fanMeiMontanari2021,celentanoFanMei2021},
is based on minimizing--- or attempting to minimize ---the so-called TAP free energy
\begin{equation}
\label{eq:TAP-1}
    \cF(\bbm;\bY)   
        := 
        - \frac{\lambda}{2n} \bbm^\top \bY \bbm
        - \frac1n \sum_{i=1}^n \sh(m_i) 
        - \frac{\lambda^2}{4}(1 - Q(\bbm))^2,
\end{equation}
over the domain $\bbm \in [-1,1]^n$,
where $\sh(m) := -\frac{1+m}{2}\log \frac{1+m}2 - \frac{1-m}{2} \log \frac{1-m}2$ is the binary entropy function. 
The TAP free energy can be seen as a correction to a `na\"ive mean field' free energy which was first derived in the context of spin-glasses in statistical physics \cite{thouless1977}.
The TAP free energy is related to the AMP algorithm Eq.~\eqref{eq:AMP-FMM}:
the fixed-points of Eq.~\eqref{eq:AMP-FMM} 
exactly correspond to the stationary points of $\cF(\,\cdot\,;\bY)$.
In fact, approximate message passing algorithms were first derived as an iterative construction of solution to the TAP equation $\nabla \cF(\bbm;\by) = 0$ \cite{Bolthausen2014}.
It is widely believed that the global minimizer of the TAP free energy asymptotically achieves Bayes optimal performance,
and the papers \cite{fanMeiMontanari2021,celentanoFanMei2021} have made progress towards making this conjecture rigorous.
\cite{fanMeiMontanari2021} establishes that for sufficiently large $\lambda$,
all stationary points of the TAP free energy with sufficiently small objective value are clustered in a small region around the Bayes estimate, and so asymptotically achieve Bayes optimal performance.
The paper \cite{celentanoFanMei2021} showed that for all $\lambda > 1$, 
the TAP free energy has a unique local minimizer close to the Bayes estimate, 
and established its global minimality subject to a conjecture that was checked numerically. 
Thus, 
there is strong evidence that minimizing the TAP free energy achieves asymptotically optimal performance throughout the weak-recovery regime $\lambda > 1$. 
One potential advantage of the TAP formulation, however, is that it opens the possibility to use methods like gradient descent or other descent algorithms to construct estimates of $\bx$. 
Our purpose here is not to give a complete account of the origin and merits of the TAP approach.
We refer the interested reader to the discussions in \cite{thouless1977,mezard1987spin,fanMeiMontanari2021,celentanoFanMei2021} and references therein.
Rather, for our purposes, 
the TAP free energy serves as an example of the use of the Sudakov-Fernique post-AMP inequality to probe the local structure of a non-convex objective around the iterates of an AMP algorithm. 
Indeed, our approach will reproduce some of the results of \cite{celentanoFanMei2021}, but arguably in a simpler way.

More recently, 
the paper \cite{AlaouiMontanariSelke2022} considered the problem of sampling from the Sherrington-Kirkpatrick Hamiltonian.
In particular, for a matrix $\bW \sim \GOE(n)$,
one would like to sample from the measure on $\{-1,1\}^n$ with density relative to he uniform measure
\begin{equation}
    \mu_{\bW}(\bx)
    =
    \frac1{Z(\lambda,\bW)} e^{(\lambda/2)\< \bx , \bW \bx \>},
\end{equation}
where $Z(\lambda,\bW)$ is a normalizing constant.
The authors propose a randomized algorithm which runs in polynomial time, takes as input the matrix $\bW$, and outputs a random vector $\bx^{\alg} \in \{-1,1\}^n$ with distribution, conditional on $\bW$, denoted by $\mu_{\bW}^{\alg}$.
For an appropriately chosen distance $\dist$ on probability measures (in this case, the $\ell_2$-Wasserstein distance, which we will define in Section \ref{sec:notation}),
the authors show that $\dist(\mu_\bW,\mu_\bW^\alg) \pconv 0$ when  $\lambda < 1/2$,
where the probabalistic convergence is over randomness in $\bW$ and auxiliary randomness in the algorithm. 
Further, for $\lambda > 1$, 
they present an impossibility result: 
no algorithm which satisfies a certain stability property is able to sample from $\mu_{\bW}$ in the sense $\dist(\mu_\bW,\mu_\bW^\alg) \pconv 0$. 
A phase transition occurs at $\lambda = 1$ in the structure of the measure $\mu_{\bW}$ which forms the basis of this impossibility result, and which has led to the conjecture that sampling from the SK measure for $\lambda > 1$ is impossible for any polynomial time algorithm, not just those satisfying the given stability property.
Conversely, it is conjectured that for $\lambda < 1$,
sampling is possible in polynomial time, and in fact, the authors of \cite{AlaouiMontanariSelke2022} conjecture their algorithm achieves this. 
Indeed, for $\lambda \in [1/2,1)$,
their proof breaks down in only one step,
in which the local strong convexity of a certain TAP free energy, similar but distinct from the one in Eq.~\eqref{eq:TAP-1}, must be but has not been established. 

We will not in this paper describe the details of the SK-sampling algorithm, and refer the interested reader to \cite{AlaouiMontanariSelke2022} for details.
The step in which their argument breaks down occurs in a sub-routine which estimates the posterior mean of $\bx$ in the $\Z_2$-synchronization model \eqref{eq:obs} with side information \eqref{eq:init}.
Note that above, although we use $\by$ as an initialization, 
our interest was in computing the Bayes estimator given $\P(\bx \in \cdot | \bY)$;
now we will be interested in computing the mean of $\P(\bx \in \,\cdot\, | \bY , \by)$.
This can be computed,
as above, using an AMP algorithm or by minimizing a certain TAP free energy whose stationary points agree with the fixed points of the AMP algorithm.
Because we have a different target, 
the AMP algorithm used by \cite{AlaouiMontanariSelke2022} is modified to
\begin{equation}\label{eq:AMP-AMS}
\begin{aligned}
    \bz^0 &= \by,
    \qquad\qquad\qquad&
    \bbm^s &= \tanh(\bz^s),
    \\
    \bz^{s+1} &= \lambda \bY \bbm^s + \by - \lambda \sbhat_s\bbm^{s-1},
    \qquad\qquad\qquad
    &
    \sbhat_s &= \lambda(1 - Q(\bbm^s)),
\end{aligned}
\end{equation}
and the TAP free energy they use is modified to
\begin{equation}
\label{eq:TAP-2}
    \cF(\bbm;\bY)   
        := 
        - \frac{\lambda}{2n} \bbm^\top \bY \bbm
        - \frac1n \sum_{i=1}^n \sh(m_i) 
        - \frac1n \< \by,\bbm\>
        - \frac{\lambda^2(1-q_\infty)(1+q_\infty-2Q(\bbm))}{4},
\end{equation}
where $q_\infty$ is a deterministic constant which we will specify below in Proposition \ref{prop:Z2-se}.
We will show that with high-probability as $n \rightarrow \infty$,
the AMP algorithm \eqref{eq:AMP-AMS} in finitely many iterations arrives in a region of space in which the TAP free energy \eqref{eq:TAP-2} is locally strongly convex and contains a unique stationary point. 
We establish this for any $\lambda > 0$. 
For $\lambda \in [1/2,1)$,
this closes the only gap in the proof of \cite{AlaouiMontanariSelke2022} and confirms that in this regime,
their algorithm samples from the SK Gibbs measure in polynomial time in the sense of 
$\dist(\mu_\bW,\mu_\bW^\alg) \pconv 0$.

\subsection{Overview of results}
\label{sec:overview}

We now provide a brief overview of our results.

\begin{enumerate}

    \item We prove a comparison inequality--- \emph{the Sudakov-Fernique post-AMP inequality} ---which allows one to study properties of an optimization landscape defined in terms of the iterates of an approximate message passing algorithm.
    Strictly speaking, this is not a new comparison inequality: 
    in fact, it is nothing but the original Sudakov-Fernique inequality applied to a matrix whose distribution conditional on the AMP iterates is drawn from a Gaussian Orthogonal Ensemble.
    Moreover, the idea of applying the Sudakov-Fernique inequality conditionally on linear constraints is not new; for example, it was used extensively in \cite{celentanoFanMei2021}.
    Our main insight is to identify certain structure in the conditional inequality which facilitates its analysis, 
    and to write the inequality in a way which better elucidates this structure.
    The structure we identify is as follows: 
    whereas the original Sudakov-Fernique inequality involves an optimization involving a single high-dimensional Gaussian vector,
    the Sudakov-Fernique post-AMP inequality involves an optimization involving a finite number of high-dimensional approximately-Gaussian vectors.
    This structure greatly simplifies the analysis of the Sudakov-Fernique post-AMP inequality.

    \item We provide a novel proof that the TAP free energy of \cite{fanMeiMontanari2021} is strongly convex in a local region around the AMP fixed point. 
    Moreover, we show that for sufficiently large (but fixed) $k$, 
    the $k^\text{th}$ iterate of the AMP algorithm arrives in this local region of strong convexity with asymptotically high-probability.
    This result already appeared in \cite{celentanoFanMei2021},
    but our analysis simplifies several key steps, 
    and in particular,
    avoids using the Kac-Rice formula.
    In fact, the structure described above allows us to reduce the analysis of the Sudakov-Fernique inequality conditioned on a large number of AMP iterates to a certain low-dimensional optimization problem which was already analyzed in \cite{celentanoFanMei2021}.

    \item We apply our proof technique to the modified TAP free energy and AMP iteration considered in \cite{AlaouiMontanariSelke2022},
    establishing for the first time the local strong convexity of their TAP free energy around sufficiently late iterates of the AMP algorithm.
    This confirms their conjecture that the algorithm they present successfully samples from the Sherrington-Kirkpatrick Gibbs measure for all $\lambda < 1$.

    We are able to analyze the TAP free energy of \cite{celentanoFanMei2021} and of \cite{AlaouiMontanariSelke2022} in almost exactly the same way,
    and, in fact, our presentation of the two arguments is unified. 
    We feel that the ability to seamlessly analyze both problems at once is emblematic of the simplicity of the Sudakov-Fernique post-AMP approach we develop.

\end{enumerate}

\subsection{Related literature}
\label{sec:lit}

The Sudakov-Fernique inequality was developed in \cite{sudakov1971gaussian,sudakov1979geometric, fernique1975regularite},
which built on the Slepian inequality of \cite{slepian1962}.
The Sudakov-Fernique inequality was used extensively in the analysis of the TAP free energy in \cite{celentanoFanMei2021},
and the analysis in this paper is inspired by the analysis which occurs there. 
\cite{celentanoFanMei2021} apply the Sudakov-Fernique inequality after conditioning on stationarity of the TAP free energy at a point $\bbm$ chosen \emph{a priori}.
The TAP free energy in general has exponentially many stationary points,
and the authors are interested in local convexity around stationary points $\bbm$ which have been chosen \emph{a posteriori} using the AMP algorithm.
Thus, sophisticated arguments involving the Kac-Rice formula are required to show that the TAP free energy landscape around stationary points chosen a priori behave likes the landscape around points chosen a posteriori in this way. 
We avoid this problem by conditioning on the full sequence of AMP iterates, and thus avoid the Kac-Rice formula. 
Nevertheless, some keys steps in our argument, which we will point out in our proofs, are borrowed directly from \cite{celentanoFanMei2021}.

The use of the Sudakov-Fernique inequality to study optimization problems in spiked $\GOE$ models is related to the use of Gordon's comparison inequality to study optimization problems in regression models.
Gordon's inequality was first developed in \cite{Gordon1985,gordon1988},
and has recently been applied to study regression estimates in \cite{stojnic2013,thrampoulidis2015a,thrampoulidis2015b,miolane2021,celentanoMontanari2021}.
The paper \cite{celentanoMontanari2021} used Gordon conditionally on the stationary point of an optimization problem, and arrived at a conditional comparison inequality with similar structure to that of the Sudakov-Fernique post-AMP inequality.

The sharp characterization of Approximate Message Passing algorithms using state evolution was developed in \cite{Bolthausen2014,bayatiMontanari2010,javanmardMontanari2013,berthier2019}.
The Bayes risk in the $\Z_2$-synchronization problem and the Bayes optimality of the corresponding AMP algorithm was first studied by \cite{deshpande2016},
and the paper \cite{montanariVenkatarmanan2021} extended the analysis to AMP with spectral initialization.
The TAP free energy was first introduced by physicists as a variational formula for the free energy in the Sherrington-Kirkpatrick model \cite{thouless1977}. 
Substantial literature has studied the relationship between the TAP free energy and the Sherrington-Kirkpatrick Gibbs measure and $\Z_2$-synchronization posterior.
We refer the reader to the discussions in \cite{fanMeiMontanari2021,celentanoFanMei2021} for pointers to the relevant literature.
As we have already described, this paper builds on the paper \cite{celentanoFanMei2021,AlaouiMontanariSelke2022},
which study the local convexity of the TAP free energy in the $\Z_2$-synchronization model and sampling from the Sherrington-Kirkipatrick Hamiltonian, respectively.

\subsection{Notation}
\label{sec:notation}

For a sequence of random variables $X_n$,
we denote
\begin{equation}
    \pliminf_{n \rightarrow \infty}
        = \sup \Big\{ 
            t \in \reals 
            \Bigm| 
            \P(X_n \leq t) \rightarrow 0
        \Big\},
    \qquad 
    \plimsup_{n \rightarrow \infty}
        = \inf \Big\{ 
            t \in \reals 
            \Bigm| 
            \P(X_n \geq t) \rightarrow 0
        \Big\}.
\end{equation}
We use $\,\pconv\,$ to denote convergence in probability, 
and for a constant $c \in \reals$, we write $\plim_{n \rightarrow \infty} X_n = c$ to mean $X_n \pconv c$.
Recall that for probability measures $\mu,\nu$ on $\reals^m$,
a coupling of $\mu$ and $\nu$ is a probability measure $\Pi$ on $(\reals^m)^2$,
whose first $m$ coordinates are distributed according to $\mu$,
and whose last $m$ coordinates are distributed according to $\nu$.
We denote by $W_2(\reals^m)$ the space of square integrable probability measures on $\reals^m$ endowed with the $\ell_2$-Wasserstein metric,
\begin{equation}
    W_2(\mu,\nu)
        :=
        \inf_{\Pi}
            \E[\| \bX - \bY \|_2^2]^{1/2},
\end{equation}
where the infimum is over couplings $\Pi$ of $\mu$ and $\nu$, and $(\bX,\bY) \sim \Pi$.
There is a mild abuse in notation--- $W_2(\reals^m)$ denotes the metric space, and $W_2(\mu,\nu)$ denotes the metric ---but this should cause no confusion.
For a sequence of probability measures $\mu_n \in W_2(\reals^m)$,
we write $\mu_n \Wconv \mu$ to mean $W_2(\mu_n,\mu) \rightarrow 0$.
For a square matrix $\bGamma \in \reals^{m \times m}$ or a bi-infinite matrix $\bGamma \in \reals^{\Z_{>0} \times \Z_{> 0}}$,
we use $\bGamma_{\leq k}$ to denote the $k\times k$ matrix consisting of its first $k$ rows and columns, $\bGamma_{k,\,\cdot\,}$ to denote its $k^\text{th}$ row, and $\bGamma_{\,\cdot\,,k}$ to denotes its $k^\text{th}$ column, and $K_{i,j}$ to denote its $(i,j)^{\text{th}}$ entry.
For a matrix $\bM \in \reals^{n \times m}$, 
we will typically denote columns of $\bM$ with lower-case bold-face, as in $\bbm_i$ for the $i^\text{th}$ column,
and rows of $\bM$ with lower-case bold-face with a breve accent, as in $\brevebm_i$ for the $i^\text{th}$ row.
In any case, we will specify how we denote rows and columns in the text when this comes up.
For a collection of vectors and matrices, we will use $\muhat$ with these vectors and matrices as subscripts to denote the joint empirical distribution of their coordinates and rows. 
For example, if $\ba \in \reals^n$ and $\bA \in \reals^{n \times m}$,
then $\muhat_{\ba,\bA}$ is the distribution on $\reals^{m + 1}$ given by
\begin{equation}
    \muhat_{\ba,\bA} := \frac1n \sum_{i=1}^n \delta_{a_i,\breve\ba_i},
\end{equation}
where $\delta_{a_i,\breve\ba_i}$ is a unit-sized point mass at $(a_i,\breve\ba_i)$.
We will denote by $\S_{>0}^m$ the space of $m\times m$ symmetric positive-definite matrices.
We denote by $\S_{>0}^\infty$ the space of bi-infinite symmetric matrices whose finite diagonal sub-matrices are all positive definite.
For a matrix $\bA$,
we denote by $\| \bA \|_{\op}$ its operator norm.
For a vector $\bx \in \reals^n$,
we denote by $\| \bx \|$ its Euclidean norm.  
For any $n$, we denote by $\sB_2(0,R)$ the Euclidean ball in $\reals^n$ with radius $R$ centered at the origin: $\sB_2(0,R) = \{ \bx \in \reals^n \mid \| \bx \| \leq n\}$.
The dimension will always be clear from context.

\section{The Sudakov-Fernique post-AMP inequality}
\label{sec:SF-post-AMP}

First, we recall the following form of classical Sudakov-Fernique inequality 
\cite{sudakov1971gaussian,sudakov1979geometric, fernique1975regularite}.
\begin{proposition}[Sudakov-Fernique inequality]\label{lem:slepian}
    Let $\bW \sim \GOE(n)$,
    $\bxi \sim \normal(\bzero,\id_n)$,
    and for $\bv \in \reals^n$,
    let $\bg(\bv) := \| \bv \| \bxi$.
    Then for any $t \in \reals$,
    closed set $K \subset \reals^{n+m}$,
    and continuous function $f:\reals^{n+m} \rightarrow \reals$,
    \begin{equation}
    \label{eq:SF}
        \P\left(\sup_{(\bv,\bu) \in K} \bv^\top \bW \bv + f(\bu,\bv) \geq t\right)
        \leq \P\left(\sup_{(\bv,\bu) \in K} \frac{2}{\sqrt{n}}
        \langle \bg(\bv) , \bv \rangle+f(\bu,\bv) \geq t\right).
    \end{equation}
    (Here $\bv \in \reals^n$ and $\bu \in \reals^m$ for some non-negative integer $m$).
\end{proposition}
\noindent Typically we would like to understand the optimization problem on the left-hand side of Eq.~\eqref{eq:SF},
but this is difficult because it involves a $\GOE(n)$ matrix.
The utility of the Sudakov-Fernique inequality is that it allows us to instead study the optimization problem on the right-hand side of Eq.~\eqref{eq:SF}, which is substantially easier because it involves a random vector. 
The Sudakov-Fernique inequality states that this simpler optimization problem is stochastically larger than the complicated one involving a $\GOE(n)$ matrix.

An important insight of this paper is that the comparison inequality we arrive at by applying Sudakov-Fernique after a finite number of iterations of an approximate message passing algorithm has a structure remarkably similar to that in Eq.~\eqref{eq:SF}.
Although we will apply these results to the particular AMP algorithms in Eqs.~\eqref{eq:AMP-FMM} and \eqref{eq:AMP-AMS},
we will state the Sudakov-Fernique post-AMP inequality for a more general AMP algorithm,
as appears in, for example, \cite{javanmardMontanari2013}.
We will then specialize it to AMP for the $\Z_2$-synchronization problem.

\subsection{Sudakov-Fernique post-AMP: the general case}

An AMP algorithm is defined by a sequence of $L_s$-Lipschitz functions $f_s : \reals^{s+1} \rightarrow \reals$ for some sequence of Lipschitz constants $L_s$. 
The algorithm takes the form
\begin{equation}
\label{eq:AMP-gen}
\begin{aligned}
    \bg^0 &\sim \normal(0,\id_n),
    \qquad\qquad\qquad&
    \bbm^s &= f_s(\bg^0,\ldots,\bg^s),
    \\
    \bg^{s+1} &= \bW \bbm^s - \sum_{j=1}^s \sb_{sj} \bbm^{j-1},
\end{aligned}
\end{equation}
where $\sb_{sj}$ are carefully chosen deterministic constants which we will define below.
The functions $f_s$ are applied row-wise to $\bg^0,\ldots,\bg^s$;
that is, $f_s(\bg^0,\ldots,\bg^s)_i = f_s(g^0_i,\ldots,g^s_i)$.
The iterations in Eqs.~\eqref{eq:AMP-FMM} and \eqref{eq:AMP-AMS} are not in the form given by Eq.~\eqref{eq:AMP-gen},
but are well approximated by an iteration in this form,  
so that, using an argument that is standard in the AMP literature, our main result will extend to these iterations.

The behavior of the iterates of the AMP algorithm are well described by a certain scalar iteration called \emph{state evolution}.
State evolution is given by a bi-infinite covariance matrix $\bK \in \S^\infty_{\geq0}$
defined inductively by
\begin{equation}
\label{eq:se-gen}
    K_{s+1,t+1} = \E[f_s(G_0,\ldots,G_s)f_t(G_0,\ldots,G_t)],
\end{equation}
where $(G_1,\ldots,G_t) \sim \normal(0,\bK_{\leq t})$ independent of $G_0 \sim \normal(0,1)$.
Note that the base case of this definition is handled automatically: $K_{1,1} = \E[f_0(G_0)^2]$, where $G_0 \sim \normal(0,1)$.
Then, for $1 \leq j \leq s$, 
\begin{equation}
\label{eq:onsag-coef}
    \sb_{sj} := \E[\partial_{G_j} f_s(G_0,\ldots,G_s)],
\end{equation}
where $(G_1,\ldots,G_t)$ and $G_0$ have the distribution given by the state evolution.
Here, the derivative exists almost surely by Rademacher's theorem.

The state evolution for AMP is given by the following theorem, proved, for example, in \cite{javanmardMontanari2013}.
\begin{proposition}[State evolution]
\label{prop:se-gen}
    For $\bW \sim \GOE(n)$ and any $k \in \Z_{>0}$,
    the empirical distribution of the coordinates of the AMP iteration converges almost surely 
    in $W_2(\reals^{k+1})$ as follows:
    \begin{equation}
        \frac1n \sum_{i=1}^n \delta_{(g_i^0,\ldots,g_i^{k+1})}
            \xrightarrow[n \rightarrow \infty]{W_2}
            \normal(0,1) \times \normal(0,\bK_{\leq k}).
    \end{equation}
    In particular, if $G_0 \sim \normal(0,1)$ independent of $(G_1,\ldots,G_k) \sim \normal(0,\bK_{\leq k})$ and, for $0 \leq s \leq k$, $M_s := f_s(G_0,\ldots,G_s)$,
    we get almost surely
    \begin{equation}
        \frac1n \sum_{i=1}^n \delta_{(m_i^0,\ldots,m_i^k,g_i^0,\ldots,g_i^{k+1})}
            \xrightarrow[n \rightarrow \infty]{W_2}
            (M_0,\ldots,M_k,G_0,\ldots,G_k).
    \end{equation}
\end{proposition}
Proposition \ref{prop:se-gen} states that, in a certain sense,
the iterates $\bg^0,\ldots,\bg^k$ behave like high-dimensional jointly Gaussian vectors with a covariance described by the state evolution matrix $\bK_{\leq k}$.
A consequence of Proposition \ref{prop:se-gen} and the definition of $\bK$ in Eq.~\eqref{eq:se-gen} is that the geometry of the iterates $\bbm^s$ agrees with the geometry of the iterates $\bg^{s+1}$.
In particular,
if we let $\bM \in \reals^{n \times k}$ be the matrix whose $s^\text{th}$ column is $\bbm^{s-1}$,
and $\bG \in \reals^{n \times k}$ be the matrix whose $s^\text{th}$ column is $\bg^s$, 
we have almost surely
\begin{equation}
    \lim_{n \rightarrow \infty} \frac1n\bM^\top \bM
        =
        \lim_{n \rightarrow \infty} \frac1n\bG^\top \bG
        =
        \bK_{\leq k}.
\end{equation}
Provided $\bK_{\leq k}$ is positive definite,
both $\{ \bbm^0,\ldots,\bbm^{k-1}\}$ and $\{ \bg^1,\ldots,\bg^k\}$ span $k$-dimensional subspaces of $\reals^n$, 
and there is a canonical linear mapping between them which maps $\bbm^{s-1}$ to $\bg^s$ for all $s$ which, for $k$ fixed with $n \rightarrow \infty$, is an approximate isometry.
This linear mapping is given by
\begin{equation}
\label{eq:T-AMP-gen}
    \bT := \bG (\bM^\top \bM)^{-1} \bM^\top.
\end{equation}
The mapping $\bT$ first projects onto the space spanned by $\{ \bbm^0,\ldots,\bbm^{k-1}\}$,
and then maps this space onto $\{ \bg^0,\ldots,\bg^k\}$ through an approximate isometry. 

We are now ready to state the Sudakov-Fernique post-AMP inequality.
\begin{theorem}[Sudakov-Fernique post-AMP]
\label{thm:SF-post-AMP-gen}
    Consider $\bW \sim \GOE(n)$, any $k \in \Z_{>0}$,
    and the AMP iteration in Eq.~\eqref{eq:AMP-gen}.
    Let $\bxi \sim \normal(0,\id_n)$ independent of everything else.
    Define $\bT$ as in Eq.~\eqref{eq:T-AMP-gen}, 
    and, for $\bv \in \reals^n$, define
    \begin{equation}
        \bg_{\AMP}(\bv)
            := 
            \sqrt{n} \bT \bv 
            + 
            \| \proj_{\bM}^\perp \bv \| \bxi.
    \end{equation}    
    For some $R > 0$ and $m \in \Z_{\geq 0}$,
    let $K_n = K_n(\bM,\bG)$ be a subset of $\reals^{n + m}$ contained in the ball $\sB_2(0,R)$ which may depend on the AMP iterates $\bM,\bG$, 
    and let $f_n: \reals^{n+m} \times (\reals^{n \times k})^2 \times \reals^n$ be a continuous function.
    Then,
    \begin{equation}
        \plimsup_{n \rightarrow \infty} 
        \sup_{(\bv,\bu) \in K_n} \big\{ \bv^\top \bW \bv + f_n(\bv,\bu;\bM,\bG,\bg^0) \big\}
        \leq 
        \plimsup_{n \rightarrow \infty} 
        \sup_{(\bv,\bu) \in K_n} \Big\{ \frac{2}{\sqrt{n}}\< \bg_{\AMP}(\bv) , \bv \> + f_n(\bv,\bu;\bM,\bG,\bg^0) \Big\}.
    \end{equation}
\end{theorem}
\noindent It is important that both the optimization function $f_n$ and the constraint set $K_n$ in Theorem \ref{thm:SF-post-AMP-gen} can depend on the AMP iterates $\bM,\bG$. 
As we will see in the $\Z_2$-synchronization example,
this will be essential to applying Theorem \ref{thm:SF-post-AMP-gen} to probe local properties of an optimization around the AMP iterates.
The probabilistic suprema are taken with respect to all sources of randomness, including in the choice of constraint set $K$.
We prove Theorem \ref{thm:SF-post-AMP-gen} in Appendix \ref{sec:app-proof-of-SF-post-AMP}.

It is worth stopping here to interpret Theorem \ref{thm:SF-post-AMP-gen}.
As we have described,
we can view the matrix $\bT$ as first projecting onto the space spanned by $\{ \bbm^0,\ldots,\bbm^{k-1}\}$,
and then mapping this space onto $\{ \bg^0,\ldots,\bg^k\}$ through an approximate isometry. 
Moreover, the AMP state evolution (Proposition \ref{prop:se-gen}) implies that every vector in the span of $\{ \bg^1, \ldots,\bg^k\}$ looks Gaussian in the sense of having coordinates with approximately Gaussian empirical distribution. 
Thus, $\sqrt{n}\bT \bv$ behaves like a vector drawn from $\normal(0,\| \proj_{\bM}\bv \|^2 \id_n)$.\footnote{We are not being precise about what we mean by ``behaves like,'' though a precise statement would involve the convergence of empirical coordinate distributions.}
Because $\bxi \sim \normal(0,\id_n)$ is independent of everything else,
$\bg_{\AMP}(\bv)$ behaves like a vector drawn from $\normal(0,\| \bv \|^2 \id_n)$.
Note that the vector $\bg(\bv)$ appearing in the classical Sudakov-Fernique inequality \eqref{eq:SF} also exhibits this behavior: $\bg(\bv) = \| \bv \|\bxi$ is, for fixed $\bv$, distributed $\normal(0,\| \bv \|^2 \id_n)$.
In this sense, the classical Sudakov-Fernique inequality and the Sudakov-Fernique post-AMP inequality have a remarkably similar structure.

Nevertheless, they also have important differences.
In the classical Sudakov-Fernique inequality, $\bg(\bv)$ has a direction chosen uniformly at random and independent of everything else, but has a norm determined implicitly by the optimization variable $\bv$. 
In contrast,
in the Sudakov-Fernique post-AMP inequality, 
both the direction and the norm of $\bg_{\AMP}(\bv)$ are determined implicitly by the optimization variable $\bv$.
Whereas in the classical Sudakov-Fernique inequality, $\bg(\bv)$ lies in the 1-dimensional space spanned by the Gaussian vector $\bxi$,
in the Sudakov-Fernique post-AMP inequality, $\bg_{\AMP}(\bv)$ lies in the $(k+1)$-dimensional space spanned by the approximately Gaussian vectors $\bg^1,\ldots,\bg^k$ and the Gaussian vector $\bxi$.
Its decomposition into these components exactly corresponds to the decomposition of $\bv$ into $\bbm^0,\ldots,\bbm^{k-1}$ and its component orthogonal to their span.
In this way, the Sudakov-Fernique post-AMP inequality is able to capture relationships between the AMP iterates, the optimization function $f$, the random matrix $\bW$, and the constraints set $K$.

Finally, whereas the classical Sudakov-Fernique inequality makes a comparison which is valid for finite $n$,
the Sudakov-Fernique post-AMP inequality makes a comparison which is valid asymptotically. 
Nevertheless, the proof of Theorem \ref{thm:SF-post-AMP-gen} uses a finite sample comparison inequality similar to the Sudakov-Fernique post-AMP inequality, and which provides the more fundamental comparison.
Because it is more difficult to state and harder to motivate this inequality, we defer details to Appendix \ref{sec:app-proof-of-SF-post-AMP}.

\subsection{Sudakov-Fernique post-AMP in $\Z_2$-synchronization}

The AMP algorithms in Eqs.~\eqref{eq:AMP-FMM} and \eqref{eq:AMP-AMS} are not of the type appearing in Eq.~\eqref{eq:AMP-gen}.
Nevertheless, they are well approximated by such an iteration.
Thus, it is not surprising that a result like Theorem \ref{thm:SF-post-AMP-gen} holds for this iteration.
We present this result in this section.

In this and the next section, we will present results for the AMP algorithms and TAP free energies of \cite{fanMeiMontanari2021} and \cite{AlaouiMontanariSelke2022} in a unified way. 
In particular, we consider the iteration
\begin{equation}\label{eq:AMP-Z2}
\begin{aligned}
    \bz^0 &= \by,
    \qquad\qquad\qquad&
    \bbm^s &= \tanh(\bz^s),
    \\
    \bz^{s+1} &= \lambda \bY \bbm^s + \chi_{\sx} \by - \lambda \sbhat_s\bbm^{s-1},
    \qquad\qquad\qquad
    &
    \sbhat_s &= \lambda(1 - Q(\bbm^k)).
\end{aligned}
\end{equation}
where $\bbm \in [-1,1]^n$ and $Q(\bbm) := \| \bbm \|_2^2 / n$,
and we take $\sx \in \{ \FMM, \AMS \}$,
with $\chi_{\FMM} = 0$ and $\chi_{\AMS} = 1$.
Taking $\chi_{\FMM} = 0$ gives Eq.~\eqref{eq:AMP-FMM}, the AMP algorithm considered by \cite{fanMeiMontanari2021,celentanoFanMei2021},
and taking $\chi_{\AMS} = 1$ gives Eq.~\eqref{eq:AMP-AMS}, the AMP algorithm considered by \cite{AlaouiMontanariSelke2022}.

The state evolution for iteration \eqref{eq:AMP-Z2} consists of an infinite sequence of scalar parameters $(\gamma_s)_{s \geq 0}$,
initialized at $\gamma_0$ as in Eq.~\eqref{eq:init},
defined recursively by
\begin{equation}
\label{eq:AMP-Z2-se}
\begin{aligned}
    \gamma_{s+1} &= \lambda^2 \E[\tanh^2(\gamma_s + \sqrt{\gamma_s}G)] + \chi_{\sx} \gamma_0,
\end{aligned}
\end{equation}
where $G_s \sim \normal(0,\gamma_s)$ for $s \geq 1$.
The next proposition states how these parameters describe the behavior of the iterates $\bz^s$, $\bbm^s$.
Moreover, for future reference,
we give some useful properties of the parameters $\gamma_s$ and related parameters.
\begin{proposition}[State evolution]\label{prop:Z2-se}
    Consider $\bbm^s,\bz^s$ as defined as in \eqref{eq:AMP-Z2},
    and $(\gamma_s)_{s \geq 0}$ as defined in Eq.~\eqref{eq:AMP-Z2-se}.
    Moreover,
    for $s \geq 0$,
    define
    \begin{equation}
    \label{eq:def-gs}
        \bg^{s+1} := \lambda \bW\bbm^s - \lambda^2(1-Q(\bbm^s))\bbm^{s-1}.
    \end{equation}
    Fix $k \geq 1$, 
    and let $\bM \in \reals^{n \times k}$ be the matrix whose $s^\text{th}$ column is $\bbm^{s-1}$, 
    and $\bG_{\AMP}$ be the matrix whose $s^\text{th}$ column is $\bg^s$.
    Denote the $i^\text{th}$ row of $\bM$ and $\bG_{\AMP}$ by $\brevebm_i$ and $\brevebg_i$, respectively.
    Let $\bGamma_{\leq k} \in \S^k_{\geq0}$ be the matrix given by $\Gamma_{st} = \gamma_{s \wedge t}$.
    For $\sx = \FMM$,
    we also assume $\gamma_0 < \gamma_\infty$, defined in Eq.~\eqref{eq:def-gam-infty}.

    We have the following:

    \begin{enumerate}[(i)]

        \item 
        We have $\gamma_s > 0$ for all $s \geq 0$.
        Further, $\bGamma_{\leq k}$ is positive definite.

        \item 
        Let $\bK_{\leq k} = \bGamma_{\leq k} - \chi_{\sx} \gamma_0 \ones\ones^\top$.
        Then $\bK_{\leq k}$ is positive definite.
        Further, let $X$, 
        $\brevebM = (M_0,\ldots,M_{k-1})$ and $\brevebG = (G_1,\ldots,G_k)$ be random vectors in $\reals^k$ with distribution given by 
        $\brevebG \sim \normal(0,\bK_{\leq k})$, 
        $W \sim \normal(0,1)$ independent, 
        and $M_0 = \tanh(\gamma_0X + \sqrt{\gamma_0}W)$, $M_s = \tanh(\gamma_sX + G_s + \chi_{\sx} \sqrt{\alpha}W)$.
        Then
        the AMP iterates converge almost surely in Wasserstein sense:
        \begin{equation}
            \frac1n \sum_{i=1}^n \delta_{\brevebm_i,\brevebg_i,w_i} 
                \Wconv
                (\brevebM,\brevebG,W).
        \end{equation}
        Moreover,
        \begin{equation}
            \E[\brevebG \brevebG^\top]
            =
            \lambda^2 \E[\brevebM\brevebM^\top]
            =
            \bK_{\leq k}.
        \end{equation}

        \item 
        We have for $s,t \geq 0$,
        \begin{equation}
        \begin{aligned}
            \lambda^2 \E[\tanh(\gamma_s + \indic{s > 0}G_s + \chi_{\sx} \sqrt{\gamma_0} G_0)\tanh(\gamma_t + \indic{t > 0}G_t + \chi_{\sx} \sqrt{\gamma_0} G_0)]
            &=
            \lambda^2 \E[\tanh(\gamma_{s \wedge t} + \sqrt{\gamma_{s \wedge t}}G)]
        \\
            &= \gamma_{(s \wedge t) + 1} - \chi_{\sx} \gamma_0,
        \end{aligned}
        \end{equation}
        where $G_0,G_s,G_t$ are distributed as in the previous item, and $G \sim \normal(0,1)$.

        \item 
        There exists $\gamma_\infty \in (0,1)$ such that $\gamma_s \rightarrow \gamma_\infty$ as $s \rightarrow \infty$. 
        The parameter $\gamma_\infty$ solves
        \begin{equation}
        \label{eq:def-gam-infty}
            \gamma_\infty = \lambda^2\E\big[\tanh^2(\gamma_\infty + \sqrt{\gamma_\infty} G)\big] + \chi_{\sx} \gamma_0,
        \end{equation}
        where $G \sim \normal(0,1)$.

        \item 
        Let $q_\infty = \E\big[\tanh^2(\gamma_\infty + \sqrt{\gamma_\infty} G)\big]$
        and $b_\infty = \E\big[\tanh^4(\gamma_\infty + \sqrt{\gamma_\infty} G)\big]$.
        We have the identities:
        \begin{equation}
        \begin{gathered}
            \lambda^2(1- q_\infty) < 1,
            \\
            q_\infty = \E\big[\tanh^2(\gamma_\infty + \sqrt{\gamma_\infty}G)\big] = \E\big[\tanh(\gamma_\infty + \sqrt{\gamma_\infty}G)\big],
            \\
            b_\infty = \E\big[\tanh^4(\gamma_\infty + \sqrt{\gamma_\infty}G)\big] = \E\big[\tanh^3(\gamma_\infty + \sqrt{\gamma_\infty}G)\big].
        \end{gathered}
        \end{equation}

    \end{enumerate}

\end{proposition}

We now provide the Sudakov-Fernique post-AMP inequality in the context of the AMP algorithm \eqref{eq:AMP-Z2} in the $\Z_2$-synchronization model.
\begin{corollary}[Sudakov-Fernique post-AMP in $\Z_2$-synchronization]
\label{cor:SF-post-AMP-Z2}
    Consider the data generating process in Eqs.~\eqref{eq:obs} and \eqref{eq:init},
    and the AMP iteration in Eq.~\eqref{eq:AMP-Z2}.
    Let $\bM \in \reals^{n \times k}$ be the matrix whose $s^\text{th}$ column is $\bbm^{s-1}$,
    $\bG \in \reals^{n \times k}$ be the matrix whose $s^\text{th}$ column is $\bg^s$ as defined in Eq.~\eqref{eq:def-gs},
    and $\bT = \lambda^{-1}\bG (\bM^\top \bM)^{-1} \bM^\top$.
    Let $\bxi \sim \normal(0,\id_n)$ independent of everything else,
    and 
    \begin{equation}
        \bg_{\AMP}(\bv)
            := 
            \sqrt{n} \bT \bv 
            + 
            \| \proj_{\bM}^\perp \bv \| \bxi.
    \end{equation}    
    For some $R > 0$ and $m \in \Z_{\geq 0}$,
    let $K_n = K_n(\bM,\bG)$ be a subset of $\reals^{n + m}$ which may depend on the AMP iterates $\bM,\bG$, 
    and let $f_n: \reals^{n+m} \times (\reals^{n \times k})^2$ be a continuous function.
    Then,
    \begin{equation}
        \plimsup_{n \rightarrow \infty} 
        \sup_{(\bv,\bu) \in K_n} \big\{ \bv^\top \bW \bv + f_n(\bv,\bu;\bM,\bG,\bg^0,\bx) \big\}
        \leq 
        \plimsup_{n \rightarrow \infty} 
        \sup_{(\bv,\bu) \in K_n} \Big\{ \frac{2}{\sqrt{n}}\< \bg_{\AMP}(\bv) , \bv \> + f_n(\bv,\bu;\bM,\bG,\bg^0,\bx) \Big\}.
    \end{equation}
\end{corollary}
\noindent As we show in the next section, Corollary \ref{cor:SF-post-AMP-Z2} is the central tool in establishing the local strong convexity of the TAP free energy.
The interpretation of this inequality agrees with the interpretation described in the discussion following Theorem \ref{thm:SF-post-AMP-gen}.
We prove Corollary \ref{cor:SF-post-AMP-Z2} in Appendix \ref{sec:app-proof-of-SF-post-AMP}.

\section{Local convexity of the TAP free energy}
\label{sec:main}

We expect the Sudakov-Fernique post-AMP inequality to be a generally useful tool for probing local properties of an optimization landscape around the iterates of an AMP algorithm.
We illustrate this by establishing the local strong convexity of the TAP free energy after several iterations of the AMP algorithm \eqref{eq:AMP-Z2} in the $\Z_2$-synchronization model.
As with the AMP algorithms, we will analyze the TAP free energies in Eqs.~\eqref{eq:TAP-1} and \eqref{eq:TAP-2} in a unified way.
Thus,
for $\sx \in \{ \FMM, \AMS \}$, 
define the TAP free energy
\begin{equation}
    \cF_{\sx}(\bbm) 
        := 
        -\frac{\lambda}{2n} \bbm^\top\bY\bbm
        -\cL_{\sx}(\bbm).
\end{equation} 
where
\begin{equation}
\begin{aligned}
    \cL_{\FMM}(\bbm)
        &:= 
        \frac1n \sum_{i=1}^n \sh(m_i) 
        +\frac{\lambda^2}4(1 - Q(\bbm))^2,
    \\
    \cL_{\AMS}(\bbm)
        &:= 
        \frac1n \sum_{i=1}^n \sh(m_i) 
        +
        \frac{\< \by , \bbm\>}{n}
        +
        \frac{\lambda^2(1-q_\infty)(1+q_\infty-2Q(\bbm))}4,
\end{aligned}
\end{equation}
for $\bbm \in [-1,1]^n$ and $\sh(m) := -\frac{1+m}{2}\log\frac{1+m}{2} - \frac{1-m}2 \log\frac{1-m}2$ the binary entropy function.
Taking $\sx = \FMM$ gives the TAP free energy in Eq.~\eqref{eq:TAP-1},
and $\sx = \AMS$ gives the TAP free energy in Eq.~\eqref{eq:TAP-2}.

Note the TAP free energy is $c/n$-strongly convex in an $\epsilon$-ball around the $(k-1)^\text{st}$ AMP iterate if and only if
\begin{equation}
    \inf_{ \substack{\| \bv \|_2 = 1 \\ \| \bu - \bbm^{k-1} \|/\sqrt{n} \leq \epsilon } }
    \bv^\top \nabla^2 \cF_{\sx} (\bu) \bv
    > c/n.
\end{equation}
Our next theorem establishes that for large enough $k$, this occurs with high-probability.
\begin{theorem}[Local convexity]
\label{thm:local-strong-cvx}
    Consider data generated from the $\Z_2$-synchronization model Eqs.~\eqref{eq:obs} and \eqref{eq:init}.
    Consider $\sx \in \{\FMM,\AMS\}$.
    We have the following.
    \begin{enumerate}[(i)]

        \item 
        If either (1) $\sx = \FMM$ and $\lambda > 1$, $\gamma_\infty > \gamma_0 > 0$, or (2) $\sx = \AMS$ and $\lambda > 0$, $\gamma_0 > 0$,
        then
        there exists $c,c' > 0$ (depending on $\lambda,\gamma_0$) such that,
        for any $\epsilon < c'$,
        and any sufficiently large $k$ (where ``sufficiently large'' may depend on $\lambda,\gamma_0,\epsilon$),
        \begin{equation}
        \label{eq:local-cvx-hp}
            \pliminf_{n \rightarrow \infty} 
            \inf_{ \substack{\| \bv \|_2 = 1 \\ \| \bu - \bbm^{k-1} \|/\sqrt{n} \leq \epsilon } }
                \bv^\top \nabla^2 \cF_{\sx} (\bu) \bv
                > c/n, 
        \end{equation}
        where $\bbm^{k-1} $ is defined via Eq.~\eqref{eq:AMP-Z2} with $\chi_{\sx}$ set according to $\sx$.

        \item 
        Under the conditions of item (i),
        there exists $c,c' > 0$ such that for any $\epsilon < c'$, there exists large enough $k$ such that, with probability approaching $1$ as $n \rightarrow \infty$,
        the TAP free energy contains a unique stationary point on $\{ \| \bu - \bbm^{k-1} \|/\sqrt{n} \leq \epsilon \} \cap [-1,1]^n$.
        This stationary point in fact lies in $(-1,1)^n$.
    \end{enumerate}

\end{theorem}
\noindent We prove Theorem \ref{thm:local-strong-cvx} in the next section, with some details deferred to the appendices.

\begin{remark}
\label{rmk:spectral}
    Although Theorem \ref{thm:local-strong-cvx} is stated for a random initialization,
    we expect it to hold also for a spectral initialization. 
    A formal argument would proceed as follows: 
    one can compute the leading eigenvector $\bv_0$ (with normalization $\| \bv_0 \| = 1$) of $\bY$ to accuracy $\| \bv - \bv_0 \| < \epsilon'$ for any $\epsilon' > 0$ arbitrarily small using finitely many iterations of an AMP algorithm with independent and informative initialization.
    Having computed $\bv_0$ to this accuracy,
    one can then begin running the AMP algorithm \ref{eq:AMP-Z2} for $k$ iterations. 
    This two step procedure approximates the AMP algorithm \ref{eq:AMP-Z2} with spectral initialization with a separate AMP algorithm with random and informative initialization.
    Using this alternative AMP, 
    none of the steps in the proof of Theorem \ref{thm:local-strong-cvx} would be fundamentally altered.
    The result with spectral initialization is recovered by taking $\epsilon' \rightarrow 0$.
    We do not carry out the complete details of such an argument.
\end{remark}

\section{Proof of Theorem \ref{thm:local-strong-cvx}}
\label{sec:proofs}

Theorem \ref{thm:local-strong-cvx} is proved using the Sudakov-Fernique post-AMP inequality, as it appears in Corollary \ref{cor:SF-post-AMP-Z2}.
First, we use Corollary \ref{cor:SF-post-AMP-Z2} to reduce the analysis of the objective in Eq.~\eqref{eq:local-cvx-hp} to analysis of an auxiliary optimization problem.
Although the auxiliary problem is high-dimensional,
we can, through a sequence of approximations, reduce its analysis to that of low-dimensional optimization problem which is straightforward to understand.
This reduction and analysis forms the remainder of the proof.

By symmetry, the distribution of the infimum in Theorem \ref{thm:local-strong-cvx} is the same for all realization of $\bx$. 
Thus, throughout the rest of the paper, we assume without loss of generality
\begin{equation}
    \bx = \ones = (1,1,\ldots,1)^\top.
\end{equation}
Likewise, we replace $X$ in Proposition \ref{prop:Z2-se} with $X = 1$.

Note that, to prove Theorem \ref{thm:local-strong-cvx}(i),
it is equivalent to show that there exists $c > 0$ (depending on $\lambda,\gamma_0$) such that 
\begin{equation}
    \liminf_{\epsilon \rightarrow 0}\;\;
    \liminf_{k \rightarrow \infty}\;\;
    \pliminf_{n \rightarrow \infty} 
    \inf_{ \substack{\| \bv \|_2 = 1 \\ \| \bu - \bbm^{k-1} \|/\sqrt{n} \leq \epsilon } }
        \bv^\top \nabla^2 \cF_{\sx} (\bu) \bv
        > c/n.
\end{equation}
(Note the order of the first two $\liminf$'s matters, and cannot be exchanged).
A direct calculation gives the Hessian of $\cF_{\sx}$:
\begin{equation}
\label{eq:local-eig}
    \bv^\top\nabla^2 \cF_{\sx}(\bu)\bv
        = 
        -\frac1n
        \Big(
            \lambda \bv^\top\bW\bv
            +
            f_{\sx}(\bv,\bu)
        \Big),
\end{equation}
where
\begin{equation}
\begin{aligned}
    f_{\FMM}(\bv,\bu)
        &= 
        \frac{\lambda^2 \< \ones , \bv \>^2}{n}
        - 
        \sum_{i=1}^n \frac{v_i^2}{1 - u_i^2}
        -
        \lambda^2(1-Q(\bu))\| \bv \|^2
        +
        \frac{2\lambda^2\< \bu,\bv \>^2}{n},
    \\
    f_{\AMS}(\bv,\bu)
        &= 
        \frac{\lambda^2 \< \ones , \bv \>^2}{n}
        - 
        \sum_{i=1}^n \frac{v_i^2}{1 - u_i^2}
        -
        \lambda^2(1-q_\infty)\| \bv \|^2.
\end{aligned}
\end{equation}
Now we apply the Sudakov-Fernique post-AMP inequality:
by Corollary \ref{cor:SF-post-AMP-Z2}, it suffices to show that,
in the relevant ranges of $\lambda$,
there exists $c > 0$, 
depending only $\lambda,\gamma_0 > 0$,
such that 
\begin{equation}
\label{eq:sf-ub}
    \limsup_{\epsilon \rightarrow 0}\;\;
    \limsup_{k \rightarrow \infty}\;\;
    \plimsup_{n \rightarrow \infty} 
    \sup_{ \substack{\| \bv \|_2 = 1 \\ \| \bu - \bbm^{k-1} \|/\sqrt{n} \leq \epsilon } }
    \sF_{\sx,k}(\bv,\bu;\bR,\bG,\bxi)
    < -c/n,
\end{equation}
where 
\begin{equation}
    \sF_{\sx,k}(\bv,\bu;\bR,\bG,\bxi)
        :=
        \frac{2\lambda}{\sqrt{n}} \< \bg_{\AMP}(\bv) , \bv \> 
        + f_{\sx}(\bv,\bu),
\end{equation}
for $\bg_{\AMP}(\bv)$ as defined in Corollary \ref{cor:SF-post-AMP-Z2}.
The objective $\sF_{\sx,k}$ is substantially easier to analyze than the objective $(\bv,\bu) \mapsto \bv^\top \nabla^2 \cF_{\sx}(\bu)\bv$.
The rest of this section establishes Eq.~\eqref{eq:sf-ub}.

\subsection{Reduction to optimization on Wasserstein space}

Our first observation is that $\sF_{\sx,k}$ is permutation-invariant with respect to the coordinates of its arguments, 
so that it is in fact a function of the empirical distribution
\begin{equation}
    \muhat_{\sqrt{n}\bv,\bu,\bM,\bG,\bg^0,\bxi}
        := 
        \frac1p \sum_{i=1}^n 
        \delta_{\sqrt{n}v_i, u_i, \brevebm_i , \brevebg_i ,g^0_i,\xi_i}
        \in W_2(\reals\times (-1,1)^{k+1} \times \reals^{k+2}),
\end{equation}
where $\brevebm_i$ and $\brevebg_i$ are the $i^\text{th}$ rows of $\bM$ and $\bG$, respectively.
This empirical distribution contains $p$ atomic elements of equal size, 
so that $\sF_{\sx,k}$, viewed in this way, is not defined on all of $W_2(\reals \times (-1,1)^{k+1} \times \reals^{k+1})$.
We will define a natural extension of $\sF_{\sx,k}$ to all of $W_2(\reals \times (-1,1)^{k+1} \times \reals^{k+1})$.
It will suffice to control the optimal value of this extended objective over a certain (non-random) subset of $W_2(\reals \times (-1,1)^{k+1} \times \reals^{k+1})$.

For any $\mu \in W_2(\reals^{2k+3})$,
let $(V,U,\brevebM,\brevebG,G_0,\Xi) \sim \mu$,
where $\brevebM = (M_0,\ldots,M_{k-1})$ and $\brevebG = (G_1,\ldots,G_k)$.
Define 
\begin{equation}
\label{eq:G-AMP}
    G_{\AMP} 
        := G_{\AMP}\big( V , \brevebM , \brevebG , \Xi \big) 
        := \lambda^{-1}\brevebG^\top \E\big[\brevebM\brevebM^\top\big]^\dagger \E\big[ \brevebM V \big] 
            +
            \| \proj_{\brevebM}^\perp V \|_{L_2} \Xi, 
\end{equation}
where $\proj_{\brevebM}^\perp$  is the projector onto the orthogonal complement of the linear span of $M_0,\ldots,M_{k-1}$ in $L_2$.
Explicitly, $\proj_{\brevebM}^\perp V = V - \brevebM^\top \E\big[\brevebM\brevebM^\top\big]^\dagger \E\big[ \brevebM V \big] $.
Note that $G_{\AMP}$ is a linear combination of $G_1,\ldots,G_k$, and $\Xi$, 
with coefficients which depend only on $\mu$. 
Thus, the joint distribution of $G_{\AMP}$ and $V$ is also only a function of $\mu$. 
In particular, 
the functions
\begin{equation}
\label{eq:Leps-dist}
\begin{aligned}
    \sF^{(1)}_{\FMM,k}(\mu)
        &:=
        2\lambda \< G_{\AMP} , V \>_{L_2}
            + 
            \lambda^2 \E[V]^2 
            - 
            \E\Big[\frac{V^2}{1-U^2}\Big]
            -
            \lambda^2(1 - \| U \|_{L_2}^2)
            + 
            2\lambda^2 \< V , U \>_{L_2}^2,
    \\
    \sF^{(1)}_{\AMS,k}(\mu)
        &:=
        2\lambda \< G_{\AMP} , V \>_{L_2}
            + 
            \lambda^2 \E[V]^2 
            - 
            \E\Big[\frac{V^2}{1-U^2}\Big]
            -
            \lambda^2(1 - q_\infty),
\end{aligned}
\end{equation}
are well-defined functions of $\mu$.
It is straightforward to check that
\begin{equation}
    \sF_{\sx,k}(\bv,\bu;\bM,\bG,\bxi)
        :=
        \sF^{(1)}_{\sx,k}\big(\muhat_{\sqrt{n}\bv,\bu,\bM,\bG,\bg^0,\bxi}\big),
\end{equation}
so that, indeed, $\sF^{(1)}_{\sx,k}$ can be viewed as an extension of $\sF_{\sx,k}$.
The careful reader may wonder why we include the empirical distribution of $\bg^0$ in the argument to $\sF_{\sx,k}^{(1)}$,
because the $\sF_{\sx,k}$ does not depend on $\bg^0$.
This will become apparent in the approximations we make in the next section,
which will make use of $\bg^0$.

We will give an upper bound on Eq.~\eqref{eq:sf-ub} 
in terms of the supremum of the objective $\sF^{(1)}_{\sx,k}$ taken over a particular subset of the Wasserstein space $W_2(\reals \times (-1,1)^{k+1} \times \reals^{k+1})$.  
In particular, let
\begin{equation}
    \SE_{\sx,k} := \Law(\brevebM,\brevebG,G_0,\Xi)
        \quad 
        \text{where}
        \quad
    \begin{cases}
        &\brevebG \sim \normal(0,\bK_{\leq k})
        \;\; 
        \text{independent of}
        \;\;
        \Xi,G_0 \stackrel{\mathrm{iid}}\sim \normal(0,1),
        \\
        &M_s = \tanh(\gamma_s + \indic{s > 0}G_s + \chi_{\sx} \sqrt{\gamma_0} G_0),\;\;\;\; 0 \leq s \leq k-1.
    \end{cases}
\end{equation}
In particular, $\SE_{\sx,k}$ is the AMP state evolution up to iteration $k$, 
as described in Proposition \ref{prop:Z2-se},
augmented by the independent Gaussian noise $\Xi$.
The set over which we optimize is
\begin{equation}
\begin{aligned}
    \cS^{(1)}_{\sx,k}(\epsilon)
        :=
        \Big\{
            \Law(V,U,\brevebM,\brevebG,G_0,\Xi)
            \Bigm|
            &\| V \|_{L_2} = 1, \,
            \P(|U| < 1) = 1,\,
            \| U - M_{k-1} \|_{L_2} \leq \epsilon,\,
            (\brevebM,\brevebG,G_0,\Xi) \sim \SE_{\sx,k}
        \Big\}.
\end{aligned}
\end{equation}
We have the following lemma.
\begin{lemma}
\label{lem:det-ub}
    For any fixed $\epsilon > 0$ and $0 < k < \infty$,
    \begin{equation}
    \label{eq:det-ub}
    \begin{aligned}
        \limsup_{\epsilon \rightarrow 0}\;
        \limsup_{k \rightarrow \infty}\;
        \plimsup_{n \rightarrow \infty}
        \sup_{\substack{ \| \bv \| = 1\\ 
                             \| \bu - \bbm^{k-1} \|/\sqrt{n} \leq \epsilon } }
                \sF_{\sx,k}(\bv,\bu;\bM,\bG,\bxi)
            \leq 
            \limsup_{\epsilon \rightarrow 0}\;
            \limsup_{k \rightarrow \infty}\;
            \sup_{\mu \in \cS^{(1)}_{\sx,k}(\epsilon)} \sF^{(1)}_{\sx,k}(\mu).
    \end{aligned}
    \end{equation}
\end{lemma}
\noindent We prove Lemma \ref{lem:det-ub} in Appendix \ref{sec:proof-lem-det-ub}.

\subsection{Dimensionality reduction of Wasserstein space}

The objective $\sF_{\sx,k}^{(1)}$ depends on a distribution on $(2k+4)$-dimensional space.
In particular, the complexity of studying the right-hand side of Eq.~\eqref{eq:det-ub} appears, at first glance, to grow with $k$.
In this section,
we show that for large $k$,
we may reduce the problem to studying an optimization problem over distributions on a low-dimensional space.

Consider $(V,U,\brevebM,\brevebG,G_0,\Xi) \sim \mu^{(1)}$ for $\mu^{(1)} \in \cS^{(1)}_{\sx,k}(\epsilon)$,
and define $G_{\AMP}$ as in Eq.~\eqref{eq:G-AMP}.
Note that $G_{\AMP}$, being a linear combination of jointly Gaussian random variables, is itself Gaussian. 
Because,
by Proposition \ref{prop:Z2-se},
$\E\big[\brevebG\brevebG^\top\big] = \lambda^2\E\big[\brevebM \brevebM^\top\big] = \bK_{\leq k}$,
we get
\begin{equation}
\begin{aligned}
    \E[G_{\AMP}^2]
        &=
        \E\big[V \brevebM^\top \big]^\top 
        \E\big[\brevebM \brevebM^\top\big]^{-1}
        \E\big[\brevebM V\big]
        + 
        \big\| \proj_{\brevebM}^\perp V \big\|_{L_2}^2
        =
        \big\| \proj_{\brevebM} V \big\|_{L_2}^2 + \big\| \proj_{\brevebM}^\perp V \big\|_{L_2}^2
        = \E[V^2] = 1.
\end{aligned}
\end{equation}
Moreover, 
\begin{equation}
    \< G_k , G_{\AMP} \>_{L_2}
        =
        \lambda \bK_{k,\,\cdot\,} \bK_{\leq k}^{-1} \E[\brevebM V]
        =
        \lambda \be_k^\top \E[\brevebM V]
        = 
        \lambda \<M_k,V\>_{L_2},
\end{equation}
where $\bK_{k,\,\cdot\,}$ denotes the $k^\text{th}$ row of $\bK_{\leq k}$,
and $\be_k$ is the $k^\text{th}$ standard basis vector.
Because $\| G_k \|_{L_2} = \lambda \| M_k \|_{L_2}$ (see Proposition \ref{prop:Z2-se}),
we may write
\begin{equation}
    G_{\AMP} = \frac{\<M_k,V\>_{L_2}}{\| M_k \|_{L_2}} \frac{G_k}{\|G_k\|_{L_2}} + \| \proj_{M_k}^\perp V \|_{L_2} \Xi^{(1)},
\end{equation}
for some Gaussian $\Xi^{(1)} \sim \normal(0,1)$ independent of $G_k$.
Define $M_k = \tanh(q_k + G_k + \chi_{\sx} \sqrt{\gamma_0}G_0 )$.\footnote{This is where it is useful that we included the empirical distribution of $\bg^0$ in the argument to $\sF_{\sx,k}^{(1)}$.}
By Proposition \ref{prop:Z2-se},
\begin{equation}
    \| M_k - M_{k-1} \|_{L_2}
        =
        \lambda^{-1}\sqrt{K_{k,k} - 2 K_{k,k-1} + K_{k-1,k-1}}
        = 
        \lambda^{-1}\sqrt{q_k - q_{k-1}}
        \leq \epsilon(k),
\end{equation}
where $\epsilon(k) > 0$ is a function, depending only on $\lambda,\gamma_0$,
satisfying $\epsilon(k) \rightarrow 0$ as $k \rightarrow \infty$.
Thus, under $\mu^{(1)}$ (recalling $\mu^{(1)} \in \cS_{\sx,k}^{(1)}(\epsilon)$),
we have
$\| U - M_k \|_{L_2} \leq \epsilon + \epsilon(k)$.
Let $\mu^{(2)}$ be the distribution of $(U,V,M_k,G_k,G_0,\Xi^{(1)})$ under $\mu^{(1)}$.
By Eq.~\eqref{eq:Leps-dist},
for $\sx = \FMM$ we may write
\begin{equation}
\label{eq:def-F2-FMM}
\begin{aligned}
    &\sF^{(1)}_{\FMM,k}(\mu^{(1)})
        =
        2\lambda\Big\<
            \frac{\<M_k,V\>_{L_2}}{\| M_k \|_{L_2}} \frac{G_k}{\|G_k\|_{L_2}} + \| \proj_{M_k}^\perp V \|_{L_2} \Xi^{(1)},\,
            V
        \Big\>
        + \lambda^2 \E[V]^2
        - \E\Big[\frac{V^2}{1-U^2}\Big]
            -
            \lambda^2(1 - \| U \|_{L_2}^2)
            + 
            2\lambda^2 \< V , U \>_{L_2}^2
    \\
        &\quad\leq 
        2\lambda\Big\<
            \frac{\<M_k,V\>_{L_2}}{\| M_k \|_{L_2}} \frac{G_k}{\|G_k\|_{L_2}} + \| \proj_{M_k}^\perp V \|_{L_2} \Xi^{(1)},\,
            V
        \Big\>
    \\
        &\qquad\qquad\qquad\qquad
        + \lambda^2 \E[V]^2
        - \E\Big[\frac{V^2}{1-U^2}\Big]
            -
            \lambda^2(1 - \| M_k \|_{L_2}^2)
            + 
            2\lambda^2 \< V , M_k \>_{L_2}^2
            + 
            6\lambda^2(\epsilon + \epsilon(k))
    \\
        &\quad=: \sF_{\FMM}^{(2)}(\mu^{(2)}) + 6\lambda^2(\epsilon + \epsilon(k)),
\end{aligned}
\end{equation}
where in the inequality
we have used that $U \mapsto \| U \|_{L_2}^2$ and $U \mapsto \< V , U \>_{L_2}^2$ are $2$-Lipschitz for $\| U \|_{L_2} \leq 1$ and $\| V \|_{L_2} \leq 1$.
The final equality is taken as the definition of $\sF_{\FMM}^{(2)}$.
Note the $\sF_{\FMM}^{(2)}$ does not, as a function, depend on $k$: the $k$-dependence occurs only through the distribution of the random variables $U,V,M_k,G_k,\Xi^{(1)}$.
Thus, unlike for $\sF_{\FMM,k}^{(1)}$,
$k$ does not appear in the subscript $\sF_{\FMM}^{(2)}$.
Likewise, for $\sx = \AMS$,
we may write
\begin{equation}
\begin{aligned}
    \sF^{(1)}_{\AMS,k}(\mu^{(1)})
        &=
        2\lambda\Big\<
            \frac{\<M_k,V\>_{L_2}}{\| M_k \|_{L_2}} \frac{G_k}{\|G_k\|_{L_2}} + \| \proj_{M_k}^\perp V \|_{L_2} \Xi^{(1)},\,
            V
        \Big\>
        + \lambda^2 \E[V]^2
        - \E\Big[\frac{V^2}{1-U^2}\Big]
            -
            \lambda^2(1-q_\infty)
    \\
        &=: \sF_{\AMS}^{(2)}(\mu^{(2)}).
\end{aligned}
\end{equation}
In this case, there is no approximation error.
For both $\sx = \FMM$ and $\sx = \AMS$, 
we have related the objective $\sF_{\sx,k}^{(1)}$, which is a function of distributions on a $(2k+4)$-dimensional space,
to the objective $\sF_{\sx}^{(2)}$, 
which is a function of distributions on a $6$-dimensional space.

Because, as we have argued above,
$\mu^{(1)} \in \cS_{\sx,k}^{(1)}(\epsilon)$ implies $\| U - M_k \|_{L_2} \leq \epsilon + \epsilon(k)$,
we have that
$\mu^{(1)} \in \cS^{(1)}_{\sx,k}(\epsilon)$ implies $\mu^{(2)} \in \cS^{(2)}_{\sx,k}(\epsilon + \epsilon(k))$,
where
\begin{equation}
\begin{aligned}
    \cS^{(2)}_{\sx,k}(\epsilon)
        := 
        \Big\{
            \Law(V,U,M_k,G_k,G_0,\Xi)
            \bigm|
            &\| V \|_{L_2} = 1,\;\;
            \P(|U| > 1) = 0,\;\;
            \|U - M_k \|_{L_2} \leq \epsilon,\,
        \\
            &G_k \sim \normal(0,K_{kk}) \text{ independent of } G_0,\Xi \stackrel{\mathrm{iid}}\sim \normal(0,1),
        \\
            &M_k = \tanh(\gamma_k + G_k + \chi_{\sx} \sqrt{\gamma_0}G_0)
        \Big\},
\end{aligned}
\end{equation}
In particular,
we have shown that
\begin{equation}
\label{eq:ub-1-to-2}
    \limsup_{\epsilon \rightarrow 0}\;
        \limsup_{k \rightarrow \infty}\;
        \sup_{\mu \in \cS^{(1)}_{\sx,k}(\epsilon)} \sF^{(1)}_{\sx,k}(\mu)
        \leq 
        \limsup_{\epsilon \rightarrow 0}\;
        \limsup_{k \rightarrow \infty}\;
        \sup_{\mu \in \cS^{(2)}_{\sx,k}(\epsilon + \epsilon(k))} \sF_{\sx}^{(2)}(\mu).
\end{equation}

\subsection{Reduction to analysis at state evolution fixed point}

Now, we analyze what happens when $k \rightarrow \infty$.

Consider $\mu^{(k)} \in S^{(2)}_{\sx,k}(\epsilon + \epsilon(k))$,
and let $(V,U,M_k,G_k,G_0,\Xi) \sim \mu^{(k)}$.
Assume also $\| V \|_{L_2} = 1$.
Let $K_\infty = \gamma_\infty - \chi_{\sx} \gamma_0$,
$G_\infty = \sqrt{K_\infty/K_{kk}}\, G_k$,
and $M_\infty = \tanh(\gamma_\infty + G_\infty + \chi_{\sx} \sqrt{\gamma_0}G_0)$.
Let $\epsilon(k) > 0$ denote a quantity satisfying $\epsilon(k) \rightarrow 0$ as $k\rightarrow \infty$
and which may change at each appearance.
Because $\gamma_k \rightarrow \gamma_\infty > 0$ (see Proposition \ref{prop:Z2-se}),
we have that $\| G_k - G_\infty \|_{L_2} \leq \epsilon(k)$, $\| M_k - M_\infty \|_{L_2} \leq \epsilon(k)$, and $\| M_\infty \|_{L_2} > 0$,
whence, using $\|V\|_{L_2} = 1$,
\begin{equation}
    \Big\|
        \Big(\frac{\<M_k,V\>_{L_2}}{\| M_k \|_{L_2}} \frac{G_k}{\|G_k\|_{L_2}} + \| \proj_{M_k}^\perp V \|_{L_2} \Xi\Big)
        -
        \Big(\frac{\E[M_\infty V]}{\| M_\infty \|_{L_2}} \frac{G_\infty}{\|G_\infty\|_{L_2}} + \| \proj_{M_\infty}^\perp V \|_{L_2} \Xi\Big)
    \Big\|_{L_2} \leq \epsilon(k).
\end{equation}
Thus, if we let $\mu^{(\infty)}$ be the joint distribution of $(V,U,M_\infty,G_\infty,G_0,\Xi)$,
we get $\sF_{\sx}^{(2)}(\mu^{(\infty)}) \leq \sF^{(2)}(\mu^{(\infty)}) + \epsilon(k)$.
Now, define
\begin{equation}
\begin{aligned}
    \cS^{(2)}_{\sx,\infty}(\epsilon)
        := 
        \Big\{
            \Law(V,U,M_k,G_k,W,\Xi)
            \bigm|
            &\| V \|_{L_2} = 1,\;\;
            \P(|U| > 1) = 0,\;\;
            \|U - M_\infty \|_{L_2} \leq \epsilon,\,
        \\
            &G_\infty \sim \normal(0,\lambda^2 q_\infty) \text{ independent of } G_0,\Xi \stackrel{\mathrm{iid}}\sim \normal(0,1),\,
        \\
            &M_\infty = \tanh(\gamma_\infty + G_\infty + \chi_{\sx} \sqrt{\gamma_0}G_0)
        \Big\}.
\end{aligned}
\end{equation}
Using that $K_\infty = \lambda^2 q_\infty$ (see Eq.~\eqref{eq:def-gam-infty} and the definition of $\bK_{\leq k}$) and the discussion above,
we that $\mu^{(k)} \in \cS_{\sx,k}^{(2)}(\epsilon+\epsilon(k))$ implies $\mu^{(\infty)} \in \cS_{\sx,\infty}^{(2)}(\epsilon + \epsilon(k))$. (Recall, the $\epsilon(k)$ appearing in the two expressions may be different).
We conclude
\begin{equation}
    \sup_{\mu \in \cS^{(2)}_{\sx,k}(\epsilon + \epsilon(k))} \sF_{\sx}^{(2)}(\mu)
        \leq 
        \sup_{\mu \in \cS^{(2)}_{\sx,\infty}(\epsilon + \epsilon(k))} \sF_{\sx}^{(2)}(\mu) + \epsilon(k).
\end{equation}
Because $\cS_{\sx,\infty}^{(2)}(\epsilon) \subset \cS_{\sx,\infty}^{(2)}(\epsilon')$ for $\epsilon < \epsilon'$,
we conclude
\begin{equation}
\label{eq:ub-2-to-3}
    \limsup_{\epsilon \rightarrow 0}\;
        \limsup_{k \rightarrow \infty}\;
        \sup_{\mu \in \cS^{(2)}_{\sx,k}(\epsilon + \epsilon(k))} \sF_{\sx}^{(2)}(\mu)
        \leq 
        \limsup_{\epsilon \rightarrow 0}
        \sup_{\mu \in \cS^{(2)}_{\sx,\infty}(\epsilon)} \sF_{\sx}^{(2)}(\mu).    
\end{equation}  
Combining Eqs.~\eqref{eq:ub-1-to-2}, and \eqref{eq:ub-2-to-3},
and \eqref{eq:det-ub},
\begin{equation}
\label{eq:ub-full}
    \limsup_{\epsilon \rightarrow 0}\;
    \limsup_{k \rightarrow \infty}\;
    \plimsup_{n \rightarrow \infty}
    \sup_{\substack{ \| \bv \| = 1\\ 
                         \| \bu - \bbm^{k-1} \|/\sqrt{n} \leq \epsilon } }
            \sF^{(1)}_{\sx,k}(\bv,\bu;\bM,\bG,\bxi)    
        \leq 
        \limsup_{\epsilon \rightarrow 0}
        \sup_{\mu \in \cS^{(2)}_{\sx,\infty}(\epsilon)} \sF_{\sx}^{(2)}(\mu).
\end{equation}
It now remains to bound the right-hand side of Eq.~\eqref{eq:ub-full}.

\subsection{Analysis of $\sF_{\sx}^{(2)}$}

For $\mu^{(\infty)} \in \cS_{\FMM,\infty}^{(2)}(\epsilon)$,
we have $\| M_\infty \|_{L_2}^2 = q_\infty$,
whence the term $\lambda^2(1 - \| M_\infty \|_{L_2}^2)$ which appears in the definition of $\sF_{\FMM}^{(2)}(\mu^{(\infty)})$ (see Eq.~\eqref{eq:def-F2-FMM}) 
is equal to $\lambda^2(1-q_\infty)$.
Thus,
we may write in a unified way that for $\mu^{(\infty)} \in \cS_{\sx,\infty}^{(2)}$,
we have
\begin{equation}
\begin{aligned}
    \sF_{\sx}^{(2)}(\mu^{\infty})
        &= 
        2\lambda\Big\<
            \frac{\<M_\infty,V\>_{L_2}}{\| M_\infty \|_{L_2}} \frac{G_\infty}{\|G_\infty\|_{L_2}} + \| \proj_{M_\infty}^\perp V \|_{L_2} \Xi,\,
            V
        \Big\>
    \\
        &\qquad\qquad\qquad\qquad\qquad\qquad+ \lambda^2 \E[V]^2
        - \E\Big[\frac{V^2}{1-U^2}\Big]
        -
        \lambda^2(1-q_\infty)
        +(1-\chi_{\sx})2\lambda^2\< V , M_\infty \>_{L_2}^2.
\end{aligned}
\end{equation}
This representation allows to analyze the cases $\sx = \FMM$ and $\sx = \AMS$ simultaneously.

For $\mu^{(\infty)} \in \cS_{\sx,\infty}^{(2)}(\epsilon)$, 
denote
\begin{equation}
\label{eq:def-rho-u}
    \rho = \frac{\<M_\infty, V\>}{\| M_\infty \|_{L_2}},
        \qquad 
    u = \E[V].
\end{equation}
Introducing dual variables $\alpha_\rho,\alpha_u,\alpha_v \in \reals$, 
we have
\begin{equation}
\label{eq:ub-Theta}
\begin{aligned}
    \sF_{\sx}^{(2)}(\mu^{(\infty)})
        &=
        2\lambda 
        \Big\< 
            \rho \frac{G_\infty}{\| G_\infty \|_{L_2}}
            +
            \sqrt{1 - \rho^2}\, \Xi,\,
            V
        \Big\>_{L_2}
        + 
        \lambda^2u^2
        -
        \E\Big[\frac{V^2}{1-U^2}\Big]
        -
        \lambda^2(1-q_\infty)
                + (1-\chi_{\sx})2\lambda^2q_\infty \rho^2
    \\
        &\qquad\qquad\qquad\qquad\qquad\qquad
        +\alpha_\rho\Big(\frac{\<M_\infty, V\>}{\| M_\infty \|_{L_2}} - \rho\Big)
        +\alpha_u\Big(\E[V] - u\Big)
        +\alpha_v\Big(\| V\|_{L_2}^2 - 1\Big)
    \\
        &\leq 
        \lambda^2u^2 - \lambda^2(1-q_\infty) + (1-\chi_{\sx})2\lambda^2 q_\infty \rho^2 - \alpha_\rho \rho - \alpha_u u - \alpha_v
        +
        \E\Big[\Theta(G_\infty,M_\infty,\Xi,U)\Big],
\end{aligned}
\end{equation}
where $\Theta$ is defined as
\begin{equation}
\label{eq:def-Theta}
    \Theta(g,m,\xi,u)
        :=
        \sup_{v \in \reals} 
        \Big\{
            2\lambda
            \Big(
                \frac{\rho}{\lambda q_\infty^{1/2}} g + \sqrt{1-\rho^2}\,\xi
            \Big) v
            - 
            \frac{v^2}{1-u^2}
            +
            \frac{\alpha_\rho}{q_\infty^{1/2}} mv
            +
            \alpha_u v
            +
            \alpha_vv^2
        \Big\}.
\end{equation}
We have used here that $\|G_\infty\|_{L_2}^2 = \lambda^2 q_\infty$ and $\| M_\infty \|_{L_2}^2 = q_\infty$.

Note that $\Theta(g,m,\xi,u)$ also depends on $\rho,\alpha_\rho,\alpha_u,\alpha_v$,
but we have left this dependence implicit to lighten notation.
If $\alpha_v < 1$, 
then the objective in the definition of $\Theta$ is strictly concave, 
whence the supremum is achieved at
\begin{equation}
    v = \frac{2\lambda\Big(\frac{\rho}{\lambda q_\infty^{1/2}} g + \sqrt{1-\rho^2}\,\xi\Big)+\frac{\alpha_\rho m}{q_\infty^{1/2}} + \alpha_u}{2\Big(\frac1{1-u^2} - \alpha_v\Big)},
\end{equation}
which gives the explicit expression 
\begin{equation}
\label{eq:simple-Theta}
    \Theta(g,m,\xi,u)
        =
        \frac{\Big(2\lambda\Big(\frac{\rho}{\lambda q_\infty^{1/2}} g + \sqrt{1-\rho^2}\,\xi\Big)+\frac{\alpha_\rho m}{q_\infty^{1/2}} + \alpha_u\Big)^2}{4\Big(\frac1{1-u^2} - \alpha_v\Big)}.
\end{equation}
Note that $u \rightarrow 1/(1/(1-u^2) - \alpha_v)$ is Lipschitz in $u$, with Lipschitz constant bounded away from $\infty$ for $\alpha_v$ bounded away from 1.
Thus, almost surely
\begin{equation}
    \big|
        \Theta(G_\infty,M_\infty,\Xi,U)
        -
        \Theta(G_\infty,M_\infty,\Xi,M_\infty)
    \big|
    \leq 
    L_{\alpha_v}\Big(2\lambda\Big(\frac{\rho}{\lambda q_\infty^{1/2}} G_\infty + \sqrt{1-\rho^2}\,\Xi\Big)+\frac{\alpha_\rho M_\infty}{q_\infty^{1/2}} + \alpha_u\Big)^2|U - M_\infty|,
\end{equation}
for some $L_{\alpha_v} < \infty$ when $\alpha_v < 1$.
Because $G_\infty/(\lambda q_{\infty}^{1/2})$ is standard Gaussian and $M_\infty \in [-1,1]$ (and so both have infinitely many moments),
we combine the previous display with Eq.~\eqref{eq:ub-Theta} and Cauchy-Schwartz to get that 
for any compact set $K \in \reals^2 $ and any fixed $\alpha_v$,
\begin{equation}
\begin{aligned}
    \sF^{(2)}(\mu^{(\infty)})
        \leq 
        \min_{(\alpha_\rho,\alpha_u) \in K}
        \Big\{
            \lambda^2u^2 &- \lambda^2(1-q_\infty) 
            + (1-\chi_{\sx})2\lambda^2 q_\infty \rho^2 
        \\
            &- \alpha_\rho \rho - \alpha_u u - \alpha_v
            +
                \E\Big[\Theta(G_\infty,M_\infty,\Xi,M_\infty)\Big]
            \Big\} + C\epsilon,
\end{aligned}
\end{equation}
where $C$ depends on $K$ and $\alpha_v$ (but not on $\rho,u,\epsilon$),
and we remind the reader that $\mu^{(\infty)}$ satisfies Eq.\eqref{eq:def-rho-u}.
In particular, 
for any compact set $K \in \reals^2 $,
\begin{equation}
\label{eq:wass-to-scalar}
\begin{aligned}
    &\limsup_{\epsilon \rightarrow 0}
        \sup_{\mu \in \cS^{(2)}_{\sx,\infty}(\epsilon)} \sF_{\sx}^{(2)}(\mu)
        \leq
        \sup_{\rho,u \in [-1,1]}
        \min_{(\alpha_\rho,\alpha_u) \in K}
        \sL_{\sx}(\rho,u;\alpha_\rho,\alpha_u,\alpha_v),
\end{aligned}
\end{equation}
where
\begin{equation}
    \sL_{\sx}(\rho,u;\alpha_\rho,\alpha_u,\alpha_v)
    :=
    \lambda^2u^2 - \lambda^2(1-q_\infty) + (1-\chi_{\sx})2\lambda^2 q_\infty \rho^2 - \alpha_\rho \rho - \alpha_u u - \alpha_v
        +
        \E\Big[\Theta(G_\infty,M_\infty,\Xi,M_{\infty})\Big],
\end{equation}
and we remind the reader that $\Theta$ depends implicitly on $\rho,\alpha_\rho,\alpha_u,\alpha_v$.

We have now reduced the problem to upper-bounding a certain max-min problem of five variables similar to that derived by Lemma 4.9(a) of \cite{celentanoFanMei2021}. 
The remainder of our argument is nearly equivalent to theirs,
and in particular, their proof of \cite[Lemma 4.9(b)]{celentanoFanMei2021}.
The primary difference between our argument and theirs has occurred in the proof up to this point,
and in particular, in how this objective was derived:
\cite{celentanoFanMei2021} conditions on stationarity of the TAP free energy at a point (i.e., $\nabla \cF(\bbm) = 0$).
Because the TAP free energy has many stationary points,
there is some subtlety in arguing that the conditional distribution which results from conditioning on $\nabla \cF(\bbm) = 0$ for a $\bbm$ chosen \emph{a priori} reflects the behavior of the landscape around one of the stationary points of $\cF$ chosen \emph{a posteriori}.
To make this connection,
\cite{celentanoFanMei2021} must invoke sophisticated arguments involving the Kac-Rice formula to control the number of such points.
We condition on arguably a more complicated event--- a full sequence of AMP iterates ---but because there is only one sequence of AMP iterates for each realization of $\bW$, 
we are able to avoid the difficulties that demanded the use of the Kac-Rice formula.

To analyze this max-min problem,
we write
\begin{equation}
    \E\Big[\Theta(G_\infty,M_\infty,\Xi,M_{\infty})\Big]
        = 
        \E\left[
            \frac{4\lambda^2(1-\rho^2) + \Big(\frac{2\rho}{ q_\infty^{1/2}} G_\infty+\frac{\alpha_\rho M_\infty}{q_\infty^{1/2}} + \alpha_u\Big)^2}{4\Big(\frac1{1-M_\infty^2} - \alpha_v\Big)}
        \right],
\end{equation}
so that
\begin{equation}
    \sL_{\sx}(\rho,u;\alpha_\rho,\alpha_u,\alpha_v)
        =
        (\rho,u,\alpha_\rho,\alpha_u)^\top
        \bA^{(\alpha_v)}
        (\rho,u,\alpha_\rho,\alpha_u)^\top
        -\lambda^2(1-q_\infty) 
        - \alpha_v 
        +\lambda^2\E\left[\Big(\frac1{1-M_\infty^2} - \alpha_v\Big)^{-1}\right],
\end{equation}
where 
\begin{equation}
    \bA_{\sx}^{(\alpha_v)}
        =
        \begin{pmatrix}
            \bA^{(\alpha_v)}_{\sx,11} & \bA^{(\alpha_v)}_{12} \\[5pt]
            \bA^{(\alpha_v)}_{21} & \bA^{(\alpha_v)}_{22} 
        \end{pmatrix},
\end{equation}
where 
\begin{equation}
\begin{aligned}
    \bA^{(\alpha_v)}_{\sx,11} 
        &=
        \begin{pmatrix}
            (1-\chi_{\sx})2\lambda^2 q_\infty 
            + 
            \E\Big[(G_\infty^2/q_\infty - \lambda^2)\left(\frac1{1-M_\infty^2} - \alpha_v\right)^{-1}\Big]
            & 0 \\[5pt]
            0 & \lambda^2 
        \end{pmatrix},
    \\
    \bA^{(\alpha_v)}_{12} 
        =
        \bA^{(\alpha_v)\top}_{21} 
        &= 
        \begin{pmatrix}
            -\frac{1}{2} + \frac1{2q_\infty} \E\Big[G_\infty M_\infty\Big(\frac1{1-M_\infty^2} - \alpha_v\Big)^{-1}\Big]
            & \frac{1}{2q_\infty^{1/2}} \E\Big[G_\infty \Big(\frac1{1-M_\infty^2} - \alpha_v\Big)^{-1}\Big]
            \\[5pt]
            0 & - \frac12 
        \end{pmatrix},
    \\
    \bA^{(\alpha_v)}_{22}
        &=
        \frac14
        \begin{pmatrix}
            \frac1{q_\infty}\E\Big[M_\infty^2\Big(\frac1{1-M_\infty^2} - \alpha_v\Big)^{-1}\Big]
            & 
            \frac1{q_\infty^{1/2}}\E\Big[M_\infty\Big(\frac1{1-M_\infty^2} - \alpha_v\Big)^{-1}\Big]
            \\[5pt]
            \frac1{q_\infty^{1/2}}\E\Big[M_\infty\Big(\frac1{1-M_\infty^2} - \alpha_v\Big)^{-1}\Big]
            & 
            \E\Big[\Big(\frac1{1-M_\infty^2} - \alpha_v\Big)^{-1}\Big]
        \end{pmatrix}.
\end{aligned}
\end{equation}
Letting $b_\infty = \E[M_\infty^3] = \E[M_\infty^4]$, and recalling $q_\infty = \E[M_\infty] = \E[M_\infty^2]$ (see Proposition \ref{prop:Z2-se}(v)),
we may compute these matrices at $\alpha_v = 0$:
\begin{equation}
\begin{aligned}
    \bA^{(0)}_{\sx,11} 
        &=
        \begin{pmatrix}
            (1-\chi_{\sx})2\lambda^2 q_\infty 
            - \lambda^2(1 - q_\infty)
            + 
            \frac{1}{q_\infty}
            \E\left[G_\infty^2\left(1-M_\infty^2\right)\right]
            & 0 \\[5pt]
            0 & \lambda^2 
        \end{pmatrix},
    \\
    \bA^{(0)}_{12} 
        =
        \bA^{(0)\top}_{21} 
        &= 
        \begin{pmatrix}
            -\frac{1}{2} + \frac1{2q_\infty} \E[G_\infty M_\infty(1-M_\infty^2)]
            & -\frac{1}{2q_\infty^{1/2}} \E[G_\infty M_\infty^2]
            \\[5pt]
            0 & - \frac12 
        \end{pmatrix},
    \\
    \bA^{(0)}_{22}
        &=
        \frac14
        \begin{pmatrix}
            1 - b_\infty/q_\infty
            & 
            q_\infty^{1/2}(1 -b_\infty/q_\infty)
            \\[5pt]
            q_\infty^{1/2}(1 -b_\infty/q_\infty)
            & 
            1 - q_\infty
        \end{pmatrix}.
\end{aligned}
\end{equation}
Note that
\begin{equation}
    \bA_{22}^{(0)}
    =
    \frac14 
    \E\left[
        (1-M_\infty^2)
        \begin{pmatrix}
            M_\infty/q_\infty^{1/2} \\[5pt] 1
        \end{pmatrix}
        \begin{pmatrix}
            M_\infty/q_\infty^{1/2} \\[5pt] 1
        \end{pmatrix}^\top
    \right] \succ 0.
\end{equation}
We will further show that $\bA_{11}^{(0)} - \bA_{12}^{(0)} (\bA_{22}^{(0)})^{-1} \bA_{21}^{(0)} \prec 0$.
Using that $G_\infty \sim \normal(0,\lambda^2 q_\infty)$, $M_\infty = \tanh(\gamma_\infty + G_\infty + \chi \sqrt{\gamma_0}W)$ and that $\frac{\de}{\de y} \tanh(y) = 1 - \tanh^2(y)$, 
Gaussian integration by parts gives
\begin{equation}
\begin{aligned}
    \E[G_\infty M_\infty] &= \lambda^2 q_\infty \E[1- M_\infty^2] = \lambda^2 q_\infty (1-q_\infty),
    \\
    \E[G_\infty M_\infty^2] &= 2\lambda^2 q_\infty \E[M_\infty(1-M_\infty^2)] = 2\lambda^2 q_\infty (q_\infty - b_\infty),
    \\
    \E[G_\infty M_\infty^3] &= 3\lambda^2 q_\infty  \E[M_\infty^2(1-M_\infty^2)] = 3\lambda^2 q_\infty(q_\infty - b_\infty),
    \\
    \E[G_\infty^2(1-M_\infty^2)] &= \lambda^2 q_\infty  - \lambda^2 q_\infty  \E[M_\infty^2] - 2\lambda^2 q_\infty \E[G_\infty M_\infty(1-M_\infty^2)]
    \\
        &= 
        \lambda^2 q_\infty - \lambda^2 q_\infty^2 - 2 \lambda^4 q_\infty^2(1-4q_\infty + 3b_\infty).
\end{aligned}
\end{equation}
Thus, 
\begin{equation}
\begin{aligned}
    \bA^{(0)}_{\sx,11} 
        &=
        \begin{pmatrix}
            (1-\chi_{\sx})2\lambda^2 q_\infty
            - 2 \lambda^4 q_\infty(1 - 4q_\infty + 3b_\infty)
            & 0 \\[5pt]
            0 & \lambda^2 
        \end{pmatrix},
    \\
    \bA_{12}^{(0)}
        &= 
        \frac12
        \begin{pmatrix}
            -1 + \lambda^2(1 - 4q_\infty + 3 b_\infty)
            & -2\lambda^2q_\infty^{1/2}(q_\infty- b_\infty)
            \\[5pt]
            0 & - 1
        \end{pmatrix}.
\end{aligned}
\end{equation}
Tedious but straightforward algebra gives
\begin{equation}
    \bA_{\sx,11}^{(0)} - \bA_{12}^{(0)} (\bA_{22}^{(0)})^{-1} \bA_{21}^{(0)}
        =
        \begin{pmatrix}
            c_1 & q_\infty^{1/2}c_2 \\[5pt]
            q_\infty^{1/2}c_2 & -c_2
        \end{pmatrix},
\end{equation}
where
\begin{equation}
\begin{aligned}
    c_1    
        &=
        \frac{-(1-q_\infty) + 2 \lambda^2(1-2q_\infty + b_\infty)^2 - \lambda^4(1-2q_\infty + b_\infty)^2(1-3q_\infty+2b_\infty)}{(1 - b_\infty/q_\infty)(1 - 2q_\infty + b_\infty)} - 2 \chi_{\sx} \lambda^2 q_\infty,
    \\
    c_2
        &=
        \frac{1}{1 - 2q_\infty + b_\infty}- \lambda^2.
\end{aligned}
\end{equation}
We have $q_\infty > b_\infty$ and $1-2q_\infty + b_\infty < 1 - q_\infty < 1/\lambda^2$ (see Proposition \ref{prop:Z2-se}(v)),
whence $c_2 > 0$.
Moreover, the Schur complement is
\begin{equation}
    c_1 - q_\infty c_2^2/(-c_2) = c_1 + q_\infty c_2 
        = 
        -q_\infty \lambda^2 (1-\lambda^2(1-2q_\infty + b_\infty))
        -\frac{(1-\lambda^2(1-2q_\infty+b_\infty))^2}{1-b_\infty/q_\infty}
        - 2 \chi_{\sx}\lambda^2 q_\infty
        < 0.
\end{equation}
Because $-c_2 < 0$ and $c_1 + \gamma_\infty c_2 < 0$, 
we conclude that $\bA_{\sx,11}^{(0)} - \bA_{12}^{(0)} (\bA_{22}^{(0)})^{-1} \bA_{21}^{(0)} \prec 0$.

Because these matrices are continuous in $\alpha_v$,
we also have $\bA_{\sx,11}^{(0)} - \bA_{12}^{(0)} (\bA_{22}^{(0)})^{-1} \bA_{21}^{(0)} \prec 0$ for all $\alpha_v$ sufficiently small.
We thus get that for such $\alpha_v$,
the function $(\rho,u) \mapsto \min_{(\alpha_\rho,\alpha_u) \in K} \sL_{\sx}(\rho,u;\alpha_\rho,\alpha_u,\alpha_v)$ is strictly concave and is maximized at $(\rho,u) = (0,0)$,
at which point the minimization over $(\alpha_\rho,\alpha_u)$ is achieved at $(\alpha_\rho,\alpha_u) = (0,0)$.
That is,
for sufficiently small $\alpha_v$,
\begin{equation}
    \sup_{\rho,u \in [-1,1]}
    \min_{(\alpha_\rho,\alpha_u) \in K}
        \sL_{\sx}(\rho,u;\alpha_\rho,\alpha_u,\alpha_v)
        =
        \sL_{\sx}(0,0;0,0,\alpha_v)
        =
        -\lambda^2(1-q_\infty)
        -\alpha_v
        +\lambda^2\E\Big[\Big(\frac1{1-M_\infty^2} - \alpha_v\Big)^{-1}\Big].
\end{equation}
Note that $\sL_{\sx}(0,0;0,0,0) = -\lambda^2(1-q_\infty) + \lambda^2\E[1-M_\infty^2] = 0$,
and, by dominated convergence, 
$\partial_{\alpha_v}\sL_{\sx}(0,0;0,0,\alpha_v)\big|_{\alpha_v=0} = - 1 + \lambda^2\E[(1-M_\infty^2)^2] < -1 + \lambda^2(1-q_\infty) < 0$ (see Proposition \ref{prop:Z2-se}(v)). 
Thus, for $\alpha_v > 0$ sufficiently small, the right-hand side of \eqref{eq:wass-to-scalar} is negative.
We conclude that for some $c > 0$,
\begin{equation}
    \limsup_{\epsilon \rightarrow 0}
        \sup_{\mu \in \cS^{(2)}_{\sx,\infty}(\epsilon)} \sF_{\sx}^{(2)}(\mu)
        < -c.
\end{equation}
This completes the proof of Theorem \ref{thm:local-strong-cvx}(i). 
Theorem \ref{thm:local-strong-cvx}(ii) now follows by a straightforward argument which we present in Appendix \ref{app:unique-stat}.

\section*{Acknowledgements}

The author would like to thank Song Mei and Zhou Fan for several useful discussions and the collaboration that formed the inspiration for this work. 
The author would also like to thank Ahmed El Alaoui, Mark Sellke, and Andrea Montanari for fruitful discussions, in particular on the Sherrington-Kirkpatrick sampling problem.
The author is supported by the Miller Institute for
Basic Research in Science, University of California, Berkeley.

\bibliographystyle{amsalpha}
\bibliography{local_convexity}

\newcommand{\etalchar}[1]{$^{#1}$}
\providecommand{\bysame}{\leavevmode\hbox to3em{\hrulefill}\thinspace}
\providecommand{\MR}{\relax\ifhmode\unskip\space\fi MR }
\providecommand{\MRhref}[2]{%
  \href{http://www.ams.org/mathscinet-getitem?mr=#1}{#2}
}
\providecommand{\href}[2]{#2}
\begin{thebibliography}{SMBC{\etalchar{+}}20b}

\bibitem[AMS22]{AlaouiMontanariSelke2022}
Ahmed~El Alaoui, Andrea Montanari, and Mark Sellke, \emph{Sampling from the
  {S}herrington-{K}irkpatrick {G}ibbs measure via algorithmic stochastic
  localization}.

\bibitem[BM10]{bayatiMontanari2010}
Mohsen Bayati and Andrea Montanari, \emph{The dynamics of message passing on
  dense graphs, with applications to compressed sensing}, 2010 IEEE
  International Symposium on Information Theory, 2010, pp.~1528--1532.

\bibitem[BMN19]{berthier2019}
Raphaël Berthier, Andrea Montanari, and Phan-Minh Nguyen, \emph{{State
  evolution for approximate message passing with non-separable functions}},
  Information and Inference: A Journal of the IMA \textbf{9} (2019), no.~1,
  33--79.

\bibitem[Bol14]{Bolthausen2014}
Erwin Bolthausen, \emph{An iterative construction of solutions of the {TAP}
  equations for the {S}herrington--{K}irkpatrick model}, Communications in
  Mathematical Physics \textbf{325} (2014), no.~1, 333--366.

\bibitem[CC15]{chenCandes2015}
Yuxin Chen and Emmanuel Candes, \emph{Solving random quadratic systems of
  equations is nearly as easy as solving linear systems}, Advances in Neural
  Information Processing Systems (C.~Cortes, N.~Lawrence, D.~Lee, M.~Sugiyama,
  and R.~Garnett, eds.), vol.~28, Curran Associates, Inc., 2015.

\bibitem[CCFM19]{Chen2019}
Yuxin Chen, Yuejie Chi, Jianqing Fan, and Cong Ma, \emph{Gradient descent with
  random initialization: fast global convergence for nonconvex phase
  retrieval}, Mathematical Programming \textbf{176} (2019), no.~1, 5--37.

\bibitem[CCM21]{celentanoCheng2021}
Michael Celentano, Chen Cheng, and Andrea Montanari, \emph{The high-dimensional
  asymptotics of first order methods with random data}.

\bibitem[CFM21]{celentanoFanMei2021}
Michael Celentano, Zhou Fan, and Song Mei, \emph{Local convexity of the {TAP}
  free energy and {AMP} convergence for {Z2}-synchronization}.

\bibitem[CLM16]{caiLi2016}
T.~Tony Cai, Xiaodong Li, and Zongming Ma, \emph{{Optimal rates of convergence
  for noisy sparse phase retrieval via thresholded Wirtinger flow}}, The Annals
  of Statistics \textbf{44} (2016), no.~5, 2221 -- 2251.

\bibitem[CLS15]{candesLi2015}
Emmanuel~J. Cand{\'e}s, Xiaodong Li, and Mahdi Soltanolkotabi, \emph{Phase
  retrieval via {W}irtinger flow: Theory and algorithms}, IEEE Transactions on
  Information Theory \textbf{61} (2015), no.~4, 1985--2007.

\bibitem[CM21]{celentanoMontanari2021}
Michael Celentano and Andrea Montanari, \emph{{CAD}: Debiasing the {L}asso with
  inaccurate covariate model}.

\bibitem[CMW20]{celentanoMontanari2020}
Michael Celentano, Andrea Montanari, and Yuchen Wu, \emph{The estimation error
  of general first order methods}, Proceedings of Thirty Third Conference on
  Learning Theory (Jacob Abernethy and Shivani Agarwal, eds.), Proceedings of
  Machine Learning Research, vol. 125, PMLR, 09--12 Jul 2020, pp.~1078--1141.

\bibitem[DAM16]{deshpande2016}
Yash Deshpande, Emmanuel Abbe, and Andrea Montanari, \emph{{Asymptotic mutual
  information for the balanced binary stochastic block model}}, Information and
  Inference: A Journal of the IMA \textbf{6} (2016), no.~2, 125--170.

\bibitem[Fer75]{fernique1975regularite}
Xavier Fernique, \emph{Regularit{\'e} des trajectoires des fonctions
  al{\'e}atoires gaussiennes}, Ecole d’Et{\'e} de Probabilit{\'e}s de
  Saint-Flour IV—1974, Springer, 1975, pp.~1--96.

\bibitem[FMM21]{fanMeiMontanari2021}
Zhou Fan, Song Mei, and Andrea Montanari, \emph{{TAP free energy, spin glasses
  and variational inference}}, The Annals of Probability \textbf{49} (2021),
  no.~1, 1 -- 45.

\bibitem[GAS{\etalchar{+}}19]{goldt2019}
Sebastian Goldt, Madhu Advani, Andrew~M Saxe, Florent Krzakala, and Lenka
  Zdeborov\'{a}, \emph{Dynamics of stochastic gradient descent for two-layer
  neural networks in the teacher-student setup}, Advances in Neural Information
  Processing Systems (H.~Wallach, H.~Larochelle, A.~Beygelzimer,
  F.~Alch\'{e}-Buc, E.~Fox, and R.~Garnett, eds.), vol.~32, Curran Associates,
  Inc., 2019.

\bibitem[Gor85]{Gordon1985}
Yehoram Gordon, \emph{Some inequalities for {G}aussian processes and
  applications}, Israel Journal of Mathematics \textbf{50} (1985), no.~4,
  265--289.

\bibitem[Gor88]{gordon1988}
Y.~Gordon, \emph{On {M}ilman's inequality and random subspaces which escape
  through a mesh in $\mathbb{R}^n$}, Geometric Aspects of Functional Analysis
  (Berlin, Heidelberg) (Joram Lindenstrauss and Vitali~D. Milman, eds.),
  Springer Berlin Heidelberg, 1988, pp.~84--106.

\bibitem[JM13]{javanmardMontanari2013}
Adel Javanmard and Andrea Montanari, \emph{{State evolution for general
  approximate message passing algorithms, with applications to spatial
  coupling}}, Information and Inference: A Journal of the IMA \textbf{2}
  (2013), no.~2, 115--144.

\bibitem[MKUZ19]{manelliKrzakala2019}
Stefano~Sarao Mannelli, Florent Krzakala, Pierfrancesco Urbani, and Lenka
  Zdeborova, \emph{Passed and spurious: Descent algorithms and local minima in
  spiked matrix-tensor models}, Proceedings of the 36th International
  Conference on Machine Learning (Kamalika Chaudhuri and Ruslan Salakhutdinov,
  eds.), Proceedings of Machine Learning Research, vol.~97, PMLR, 09--15 Jun
  2019, pp.~4333--4342.

\bibitem[MKUZ20]{mignacco2020}
Francesca Mignacco, Florent Krzakala, Pierfrancesco Urbani, and Lenka
  Zdeborov\'{a}, \emph{Dynamical mean-field theory for stochastic gradient
  descent in {G}aussian mixture classification}, Advances in Neural Information
  Processing Systems (H.~Larochelle, M.~Ranzato, R.~Hadsell, M.F. Balcan, and
  H.~Lin, eds.), vol.~33, Curran Associates, Inc., 2020, pp.~9540--9550.

\bibitem[MM21]{miolane2021}
Léo Miolane and Andrea Montanari, \emph{{The distribution of the Lasso:
  Uniform control over sparse balls and adaptive parameter tuning}}, The Annals
  of Statistics \textbf{49} (2021), no.~4, 2313 -- 2335.

\bibitem[MPV87]{mezard1987spin}
M.~Mezard, G.~Parisi, and M.A. Virasoro, \emph{Spin glass theory and beyond: An
  introduction to the replica method and its applications}, World Scientific
  Lecture Notes In Physics, World Scientific Publishing Company, 1987.

\bibitem[MV21a]{montanari2021}
Andrea Montanari and Ramji Venkataramanan, \emph{{Estimation of low-rank
  matrices via approximate message passing}}, The Annals of Statistics
  \textbf{49} (2021), no.~1, 321 -- 345.

\bibitem[MV21b]{montanariVenkatarmanan2021}
\bysame, \emph{{Estimation of low-rank matrices via approximate message
  passing}}, The Annals of Statistics \textbf{49} (2021), no.~1, 321 -- 345.

\bibitem[MW22]{montanariWu2022}
Andrea Montanari and Yuchen Wu, \emph{Statistically optimal first order
  algorithms: A proof via orthogonalization}, 2022.

\bibitem[MWCC20]{Ma2020}
Cong Ma, Kaizheng Wang, Yuejie Chi, and Yuxin Chen, \emph{Implicit
  regularization in nonconvex statistical estimation: Gradient descent
  converges linearly for phase retrieval, matrix completion, and blind
  deconvolution}, Foundations of Computational Mathematics \textbf{20} (2020),
  no.~3, 451--632.

\bibitem[MZ20]{Sarao_Mannelli_2020}
Stefano~Sarao Mannelli and Lenka Zdeborov{\'{a}}, \emph{Thresholds of
  descending algorithms in inference problems}, Journal of Statistical
  Mechanics: Theory and Experiment \textbf{2020} (2020), no.~3, 034004.

\bibitem[Sle62]{slepian1962}
David Slepian, \emph{The one-sided barrier problem for gaussian noise}, The
  Bell System Technical Journal \textbf{41} (1962), no.~2, 463--501.

\bibitem[SMBC{\etalchar{+}}19]{manelliBiroli2019}
Stefano Sarao~Mannelli, Giulio Biroli, Chiara Cammarota, Florent Krzakala, and
  Lenka Zdeborov\'{a}, \emph{Who is afraid of big bad minima? analysis of
  gradient-flow in spiked matrix-tensor models}, Advances in Neural Information
  Processing Systems (H.~Wallach, H.~Larochelle, A.~Beygelzimer,
  F.~Alch\'{e}-Buc, E.~Fox, and R.~Garnett, eds.), vol.~32, Curran Associates,
  Inc., 2019.

\bibitem[SMBC{\etalchar{+}}20a]{sarao2020}
Stefano Sarao~Mannelli, Giulio Biroli, Chiara Cammarota, Florent Krzakala,
  Pierfrancesco Urbani, and Lenka Zdeborov\'{a}, \emph{Complex dynamics in
  simple neural networks: Understanding gradient flow in phase retrieval},
  Advances in Neural Information Processing Systems (H.~Larochelle, M.~Ranzato,
  R.~Hadsell, M.F. Balcan, and H.~Lin, eds.), vol.~33, Curran Associates, Inc.,
  2020, pp.~3265--3274.

\bibitem[SMBC{\etalchar{+}}20b]{manelliBiroli2020}
Stefano Sarao~Mannelli, Giulio Biroli, Chiara Cammarota, Florent Krzakala,
  Pierfrancesco Urbani, and Lenka Zdeborov\'a, \emph{Marvels and pitfalls of
  the {L}angevin algorithm in noisy high-dimensional inference}, Phys. Rev. X
  \textbf{10} (2020), 011057.

\bibitem[Sto13]{stojnic2013}
Mihailo Stojnic, \emph{A framework to characterize performance of lasso
  algorithms}.

\bibitem[Sud71]{sudakov1971gaussian}
Vladimir~Nikolaevich Sudakov, \emph{Gaussian random processes and measures of
  solid angles in {H}ilbert space}, Doklady Akademii Nauk, vol. 197, Russian
  Academy of Sciences, 1971, pp.~43--45.

\bibitem[Sud79]{sudakov1979geometric}
Vladimir~N Sudakov, \emph{Geometric problems in the theory of
  infinite-dimensional probability distributions}, vol. 141, American
  Mathematical Soc., 1979.

\bibitem[SWW17]{Sanghavi2017}
Sujay Sanghavi, Rachel Ward, and Chris~D. White, \emph{The local convexity of
  solving systems of quadratic equations}, Results in Mathematics \textbf{71}
  (2017), no.~3, 569--608.

\bibitem[TAH15]{thrampoulidis2015a}
Christos Thrampoulidis, Ehsan Abbasi, and Babak Hassibi, \emph{Precise
  high-dimensional error analysis of regularized {M}-estimators}, 2015 53rd
  Annual Allerton Conference on Communication, Control, and Computing
  (Allerton), 2015, pp.~410--417.

\bibitem[TAP77]{thouless1977}
D.~J. Thouless, P.~W. Anderson, and R.~G. Palmer, \emph{Solution of `solvable
  model of a spin glass'}, The Philosophical Magazine: A Journal of Theoretical
  Experimental and Applied Physics \textbf{35} (1977), no.~3, 593--601.

\bibitem[TOH15]{thrampoulidis2015b}
Christos Thrampoulidis, Samet Oymak, and Babak Hassibi, \emph{Regularized
  linear regression: A precise analysis of the estimation error}, Proceedings
  of The 28th Conference on Learning Theory (Paris, France) (Peter Grünwald,
  Elad Hazan, and Satyen Kale, eds.), Proceedings of Machine Learning Research,
  vol.~40, PMLR, 03--06 Jul 2015, pp.~1683--1709.

\bibitem[Vil08]{villani2008optimal}
C.~Villani, \emph{Optimal transport: Old and new}, Grundlehren der
  mathematischen Wissenschaften, Springer Berlin Heidelberg, 2008.

\bibitem[VSL{\etalchar{+}}22]{veiga2022}
Rodrigo Veiga, Ludovic Stephan, Bruno Loureiro, Florent Krzakala, and Lenka
  Zdeborová, \emph{Phase diagram of stochastic gradient descent in
  high-dimensional two-layer neural networks}.

\end{thebibliography}

\newpage
\appendix

\section{Proof of the Sudakov-Fernique post-AMP inequality}
\label{sec:app-proof-of-SF-post-AMP}

The goal of this section is to prove Theorem \ref{thm:SF-post-AMP-gen} and Corollary \ref{cor:SF-post-AMP-Z2}.
We start by providing a finite sample comparison inequality, which we call the \emph{conditional Sudakov-Fernique inequality},
in Appendix \ref{sec:conditional-SF}.
The Sudakov-Fernique post-AMP inequality results from approximating the conditional Sudakov-Fernique inequality in the limit $n \rightarrow \infty$ using AMP state evolution.
We carry out this approximation argument in Appendix \ref{sec:proof-of-SF-post-AMP}.

\subsection{The conditional Sudakov-Fernique inequality}
\label{sec:conditional-SF}

The conditional Sudakov-Fernique inequality is nothing but the classical Sudakov-Fernique inequality applied conditional on a set of linear constraints.
Thus, it is not technically a new comparison inequality.
Our insight is to write the conditional inequality in a way which better reveals its structure and facilitates its application in the context of AMP.

We apply the classical Sudakov-Fernique inequality conditional on a set of linear-constraints on $\bW$ of the form
\begin{equation}
    \bW \bR = \bS,
\end{equation}
for matrices $\bR,\bS \in \reals^{n \times k}$,
where $\bR$ is full rank.
Given any such matrices, 
let $\bB_{\SF} = \bB_{\SF}(\bR,\bS) \in \reals^{k \times k}$ be any solution to the equations
\begin{equation}
\label{eq:B-def}
    \bR^\top \bS = (\bR^\top \bR) \bB_{\SF} + \bB_{\SF}^\top (\bR^\top \bR).
\end{equation}

\begin{lemma}[Existence of solutions to Eq.~\eqref{eq:B-def}]
\label{lem:exist-unique}
    If $\bR^\top \bS$ is symmetric (i.e., $\bR^\top \bS = \bS^\top \bR$), 
    then Eq.~\eqref{eq:B-def} has solutions.
    This occurs, in particular, 
    if there exists a symmetric matrix $\bW$ with $\bW \bR = \bS$.
\end{lemma}

\begin{proof}[Proof of Lemma \ref{lem:exist-unique}]
    We may take $\bB_{\SF} = \frac12(\bR^\top \bR)^{-1}\bR^\top \bS$,
    where the inverse exists because $\bR$ is full-rank.
    Note that if $\bW \bR = \bS$ for symmetric $\bW$, then $\bR^\top \bS = \bR^\top \bW \bR = \bS^\top \bR$.
\end{proof}

The purpose of defining $\bB_{\SF}$ is to provide a simple representation of the distribution of $\bW$ conditional on the event $\bW \bR = \bS$,
which we do in the next lemma.
\begin{lemma}
\label{lem:W-det}
    Let $\bG_{\SF} = \bS - \bR \bB_{\SF}$ for any $\bB_{\SF}$ satisfying Eq.~\eqref{eq:B-def}.
    Further, let $\bT_{\SF} = \bG_{\SF}(\bR^\top \bR)^{-1}\bR^\top$ is the unique matrix satisfying $\bT_{\SF} \bR = \bG_{\SF}$ and $\bT_{\SF} \proj_{\bR}^\perp = 0$.
    Then
    \begin{equation}
        \bW \bR = \bS
        \quad 
        \text{implies}
        \quad
        \bW - \proj_{\bR}^\perp \bW \proj_{\bR}^\perp
            = 
            \bT_{\SF} + \bT_{\SF}^\top.
    \end{equation}
    Thus,
    \begin{equation}
        \Law\big(\bW \mid \bW\bR = \bS\big)
        =
        \Law\big(\bT_{\SF} + \bT_{\SF}^\top + \proj_{\bR}^\perp \bWtilde \proj_{\bR}^\perp \big),
    \end{equation}
    where $\bWtilde \sim \GOE(n)$.
\end{lemma}

\begin{proof}[Proof of Lemma \ref{lem:W-det}]
    If $\bW \bR = \bS$, 
    then
    \begin{equation}
    \begin{aligned}
        \bW - \proj_{\bR}^\perp \bW \proj_{\bR}^\perp 
            &= 
            \bW \proj_{\bR} + \proj_{\bR} \bW \proj_{\bR}^\perp
        \\
            &= 
            \bW \bR (\bR^\top \bR)^{-1}\bR^\top 
            +
            \bR (\bR^\top \bR)^{-1}\bR^\top \bW \big(\Id_n - \bR (\bR^\top \bR)^{-1}\bR^\top \big)
        \\
            &= 
            \bS (\bR^\top \bR)^{-1}\bR^\top 
            + 
            \bR (\bR^\top \bR)^{-1}\bS^\top 
            - 
            \bR (\bR^\top \bR)^{-1} \bR^\top \bS (\bR^\top \bR)^{-1}\bR^\top 
        \\
            &= 
            \bS (\bR^\top \bR)^{-1}\bR^\top 
            + 
            \bR (\bR^\top \bR)^{-1}\bS^\top 
            - 
            \bR \bB_{\SF} (\bR^\top\bR)^{-1} \bR^\top 
            -
            \bR(\bR^\top \bR)^{-1} \bB_{\SF}^\top \bR^\top
        \\
            &= \bT_{\SF} + \bT_{\SF}^\top,
    \end{aligned}
    \end{equation}
    as desired.

    Next, note that $\proj_{\bR}^\perp \bW \proj_{\bR}^\perp$ is independent of $\bW - \proj_{\bR}^\perp \bW \proj_{\bR}^\perp$, 
    and the occurrence of the event  $\bW \bR = \bS$ is a function of $\bW - \proj_{\bR}^\perp \bW \proj_{\bR}^\perp$.
    The final distributional claim of the lemma follows.
\end{proof}

We can now state the conditional Sudakov-Fernique inequality.
\begin{theorem}[Conditional Sudakov-Fernique inequality]
\label{thm:cond-SF}
    Consider $\bR,\bS \in \reals^{n \times k}$,
    and assume $\bR$ is full-rank.
    Let $\bB_{\SF} \in \reals^{k \times k}$, $\bG_{\SF},\bT_{\SF} \in \reals^{n \times k}$ be random matrices which satisfy Eq.~\eqref{eq:B-def} and the conditions of Lemma \ref{lem:W-det} almost surely.
    Define the (random) function $\bg_{\SF}: \reals^n \rightarrow \reals^n$ by
    \begin{equation}
    \label{eq:def-gSF}
        \bg_{\SF}(\bv) 
            := 
            \sqrt{n} \, \bT_{\SF} \bv + \| \proj_{\bR}^\perp \bv \| \proj_{\bR}^\perp \bxi,
    \end{equation}
    where $\bxi \sim \normal(0,\Id_n)$ independent of everything else.
    Then, for any compact set $K = K(\bR,\bS) \subset \reals^{n + m}$ (which may depend on $\bR,\bS$) and bounded and continuous $f:\reals^{m+n} \times (\reals^{n\times k})^2 \rightarrow \reals$,
    we have
    \begin{equation}
        \P\left(
            \sup_{(\bv,\bu) \in K} \big\{ \bv^\top \bW \bv + f(\bv,\bu; \bR,\bS) \big\} \geq t 
            \Bigm|
            \bW\bR = \bS
        \right)
        \leq 
        \P\left(
            \sup_{(\bv,\bu) \in K} \Big\{ \frac2{\sqrt{n}}\< \bg_{\SF}(\bv),\bv\> + f(\bv,\bu;\bR,\bS) \Big\} \geq t 
        \right).
    \end{equation}
\end{theorem}

\begin{proof}[Proof of Theorem \ref{thm:cond-SF}]
    By Lemma \ref{lem:W-det}, 
    for fixed $\bR,\bS \in \reals^{n \times k}$
    \begin{equation}
    \begin{aligned}
        &\P\left(
            \sup_{(\bv,\bu) \in S} \big\{ \bv^\top \bW \bv + f(\bv,\bu; \bR,\bS) \big\} \geq t 
            \Bigm|
            \bW \bR = \bS 
        \right)
        \\
        &\qquad\qquad =
        \P\left(
            \sup_{(\bv,\bu) \in S} \big\{ \bv^\top \proj_{\bR}^\perp \bWtilde \proj_{\bR}^\perp \bv + \frac{2}{\sqrt{n}} \< \bT_{\SF} \bv, \bv \> + f(\bv,\bu; \bR,\bS) \big\} \geq t 
            \Bigm|
            \bW \bR = \bS 
        \right).
    \end{aligned}
    \end{equation}
    Because $\bWtilde$ is the only part of the optimization in the second probability which is random,
    and it is conditionally distributed $\GOE(n)$, 
    by the Sudakov-Fernique inequality (Proposition \ref{lem:slepian}),
    the right-hand side of the previous display is bounded by
    \begin{equation}
        \mathrm{RHS}
            \leq 
            \P\left(
            \sup_{(\bv,\bu) \in S} \Big\{ \frac{2}{\sqrt{n}}\| \proj_{\bR}^\perp \bv \|_2 \< \proj_{\bR}^\perp \bxi ,  \bv\>  + \frac{2}{\sqrt{n}} \<\bT_{\SF} \bv,  \bv \> + f(\bv,\bu; \bR,\bS) \Big\} \geq t
        \right).
    \end{equation}
    Note that the objective in the previous display is nothing but $\frac2{\sqrt{n}}\< \bg_{\SF}(\bv),\bv\> + f(\bv,\bu;\bR,\bS)$,
    which gives the result.
\end{proof}

\subsection{Sudakov-Fernique post-AMP}
\label{sec:proof-of-SF-post-AMP}

Theorem \ref{thm:SF-post-AMP-gen} and Corollary \ref{cor:SF-post-AMP-Z2} follow from Theorem \ref{thm:cond-SF} 
with matrices explicitly related to the AMP iteration,
at the cost of incurring an error which vanishes as $n \rightarrow \infty$.
Our main task is to show this error is small.

\begin{proof}[Proof of Theorem \ref{thm:SF-post-AMP-gen}]
    Recall the $\bM \in \reals^{n \times k}$ is the matrix whose $s^\text{th}$ column is given by $\bbm^{s-1}$ and $\bG \in \reals^{n \times k}$ is the matrix whose $s^\text{th}$ column is given by $\bg^s$, 
    which are generated by the AMP algorithm Eq.~\eqref{eq:AMP-gen}.
    Let 
    \begin{equation}
        \bB 
        := 
        \begin{pmatrix}
            0      & \sb_{11} & \sb_{21}     & \cdots & \sb_{(k-1)1} \\[5pt]
            0      & 0     & \sb_{22} & & \sb_{(k-1)2} \\[5pt]
            \vdots &       & \ddots     & \ddots & \vdots \\[5pt]
            0 &  0     &       & 0 & \sb_{(k-1)(k-1)} \\[5pt]
            0      & 0     &    \cdots   & 0 & 0
        \end{pmatrix},
    \end{equation}
    for $\sb_{sj}$ defined in Eq.~\eqref{eq:onsag-coef}.
    We define two matrix-valued functions of the random matrix $\bW$ and noise $\bg^0$
    given by $\bS_{\AMP}(\bW,\bg^0) = \bM$ and $\bR_{\AMP}(\bW,\bg^0) = \bG + \bM \bB$,
    where we recall that $\bG,\bM$ is a function of $\bW,\bg^0$ via the AMP iteration \eqref{eq:AMP-gen}.
    Consider any fixed matrices $\bR \in \reals^{n \times k}$ and $\bS \in \reals^{n \times k}$.
    By the AMP iteration Eq.~\eqref{eq:AMP-gen},
    we may check inductively the equivalence of events (over $\bW,\bg^0$)
    \begin{equation}
        \Big\{ \bS_{\AMP}(\bW,\bg^0) = \bS,\; \bR_{\AMP}(\bW,\bg^0) = \bR \Big\}
        = 
        \Big\{ \bW \bR = \bS ,\; f(\bg_0) = \bS_{\,\cdot\,,1}\Big\}.
    \end{equation}
    Thus, conditioning on the sequence of AMP iterates $\bbm^0,\bbm^1,\ldots,\bbm^{k-1}$ and $\bg^0,\bg^1,\ldots,\bg^k$ is equivalent to conditioning on the realization of $\bg^0$, $\bS_{\AMP}(\bW)$ and $\bR_{\AMP}(\bW)$,
    which by the previous display is equivalent to conditioning on $\bg^0$ and a set of linear constraints on $\bW$ of the form handled by Theorem \ref{thm:cond-SF}.

    The key insight is that with high probability over $\bW,\bg^0$, 
    the deterministic matrix $\bB$ approximately solves Eq.~\eqref{eq:B-def}.
    Precisely,
    there exists a random matrix $\bB_{\SF} = \bB_{\SF}(\bW,\bg^0)$
    such that,
    with probability 1 over $\bW,\bg^0$,
    \begin{equation}
    \label{eq:B-SF-gen}
        \bR_{\AMP}^\top \bS_{\AMP}
            =
            (\bR_{\AMP}^\top \bR_{\AMP})\bB_{\SF}
            + 
            \bB_{\SF}^\top (\bR_{\AMP}^\top \bR_{\AMP}),
    \end{equation}
    and 
    \begin{equation}
    \label{eq:B-approx-gen}
        \| \bB_{\SF} - \bB \|_{\op} \pconv 0,
    \end{equation}
    where for compactness of notation,
    we have removed the arguments from $\bS_{\AMP}(\bW,\bg^0)$, $\bR_{\AMP}(\bW,\bg^0)$.
    Indeed,
    \begin{equation}
    \label{eq:R-S-conv}
        \frac1n
        \Big(
            \bR_{\AMP}^\top \bS_{\AMP} 
            - \bR_{\AMP}^\top \bR_{\AMP} \bB 
            - \bB \bR_{\AMP}^\top \bR_{\AMP}
        \Big)
        \pconv 0,
    \end{equation}
    because, by Proposition \ref{prop:se-gen},
    \begin{equation}
        \frac1n \bR_{\AMP}^\top \bR_{\AMP} 
            =
            \frac1n \bM^\top \bM \pconv \bK_{\leq k},
        \qquad 
        \frac1n \bR_{\AMP}^\top \bS_{\AMP}
            =
            \frac1n \bM^\top (\bG + \bM \bB)
            \pconv \bB^\top \bK_{\leq k} + \bK_{\leq k} \bB,
    \end{equation}
    where the second convergence uses Gaussian integration by parts: for $1 \leq s \leq t$,
    we have $\E[G_s M_t] = \E[G_s f_t(G_0,\ldots,G_t)] = \sum_{j=1}^t K_{sj}\sb_{sj}$,
    with $G_s,M_t$ as in Proposition \ref{prop:se-gen}.
    Thus,
    we may pick 
    \begin{equation}
        \bB_{\SF} 
            =
            \bB
                + 
                \frac12 (\bR_{\AMP}^\top \bR_{\AMP})^\dagger 
                \Big(
                    \bR_{\AMP}^\top \bS_{\AMP}
                    - 
                    \bR_{\AMP}^\top \bR_{\AMP} \bB
                    -
                    \bB^\top \bR_{\AMP}^\top \bR_{\AMP}
                \Big).
    \end{equation}
    Then Eq.~\eqref{eq:B-SF-gen} holds by definition.
    Because $\bK_{\leq k} \succ 0$,
    the second-to-last display implies Eq.~\eqref{eq:B-approx-gen}.

    Now,
    define $\bG_{\SF}$, $\bT_{\SF}$, and $\bg_{\SF}(\bv)$ 
    as in Lemma \ref{lem:W-det} and Theorem \ref{thm:cond-SF}.
    In particular, 
    define
    \begin{equation}
    \label{eq:def-gamp}
        \bG_{\SF} = \bS_{\AMP} - \bR_{\AMP} \bB_{\SF},
        \quad 
        \bT_{\SF} = \bG_{\SF}(\bR_{\AMP}^\top \bR_{\AMP})^{-1} \bR_{\AMP}^\top,
        \quad 
        \bg_{\SF}(\bv) = \sqrt{n}\, \bT_{\SF} \bv + \| \proj_{\bR_{\AMP}}^\perp \bv \| \proj_{\bR_{\AMP}}^\perp \bxi,
    \end{equation}
    where $\bxi \sim \normal(0,\id_n)$ is the same Gaussian used in the definition of $\bg_{\AMP}(\bv)$ in the statement of the theorem.
    Note the inverse $(\bR_{\AMP}^\top \bR_{\AMP})^{-1}$ exists with probability approaching 1 because $\bR_{\AMP}^\top \bR_{\AMP} / n \pconv \bK_{\geq k}$ by Proposition \ref{prop:se-gen}.
    By Eq.~\eqref{eq:B-approx-gen},
    and because  $\| \bR_{\AMP} \|_{\op}/\sqrt{n} \pconv \| \bK_{\leq k} \|_{\op}^{1/2}$ by Proposition \ref{prop:se-gen},
    we get 
    $\| \bG/\sqrt{n} - \bG_{\SF}/\sqrt{n} \|_{\op} \leq (\| \bR_{\AMP} \|_{\op}/\sqrt{n}) \| \bB - \bB_{\SF} \|_{\op} \pconv 0$.
    We then conclude that
    \begin{equation}
        \| \bT - \bT_{\SF} \|_{\op}  
            \leq 
            \Big\|
                \frac1{\sqrt{n}} \bG - \frac1{\sqrt{n}}\bG_{\SF}
            \Big\|_{\op}
            \| (\bR_{\AMP}^\top \bR_{\AMP}/n)^{-1} \|_{\op} 
            \Big\| \frac1{\sqrt{n}} \bR_{\AMP} \Big\|_{\op}
            \pconv 0,
    \end{equation}
    where we use that $\bR_{\AMP}^\top \bR_{\AMP}/n \pconv \bK_{\leq k} \succ 0$ has positive definite limit.
    Finally, note 
    $\big\| \bg_{\AMP}(\bv) - \bg_{\SF}(\bv) \big\|/\sqrt{n}
        =
        \big\| (\bT\bv - \bT_{\SF}\bv) + \| \proj_{\bR_{\AMP}}^\perp \bv \| (\proj_{\bR_{\AMP}}\bxi)/\sqrt{n} \big\|
        \leq 
        \big(\| \bT - \bT_{\SF} \|_{\op} + \| \proj_{\bR_{\AMP}} \bxi \|/\sqrt{n}\big) \| \bv \|,$
    whence
    \begin{equation}
    \begin{aligned}
        \sup_{\| \bv \| \leq R}
        \Big|
            \frac2{\sqrt{n}} \< \bg_{\AMP}(\bv) , \bv \>
            -
            \frac2{\sqrt{n}} \< \bg_{\SF} (\bv) , \bv \>
        \Big|
        &\leq \sup_{\| \bv \| \leq R} 2\|\bv\|^2
            \Big( 
                \| \bT - \bT_{\SF} \|_{\op} 
                + 
                \frac{\|\proj_{\bR_{\AMP}}\bxi\|}{\sqrt{n}}
            \Big)
        \\
        &\leq 2R^2\Big(\| \bT - \bT_{\SF} \|_{\op}+\frac{\|\proj_{\bR_{\AMP}}\bxi\|}{\sqrt{n}}\Big)
        \pconv 0.
    \end{aligned}
    \end{equation}
    
    Now Theorem \ref{thm:SF-post-AMP-gen} follows from an application of the conditional Sudakov-Fernique inequality and the approximation bound in the previous display.
    Indeed,
    applying the conditional Sudakov-Fernique inequality, 
    we have for any fixed $\bR,\bS,\bg^0$ and any $t \in \reals$ and $\epsilon > 0$,
    \begin{equation}
    \begin{aligned}
        &\P\left(
            \sup_{(\bv,\bu) \in K} \big\{ \bv^\top \bW \bv + f(\bv,\bu; \bR,\bS - \bR \bB,\bg^0) \big\} \geq t -\epsilon
            \Bigm|
            \bR_{\AMP}(\bW) = \bR, \bS_{\AMP}(\bW) = \bS, \bg^0
        \right)
    \\
        &\leq 
        \P\left(
            \sup_{(\bv,\bu) \in K} \Big\{ \frac2{\sqrt{n}}\< \bg_{\SF}(\bv),\bv\> + f(\bv,\bu;\bR,\bS - \bR \bB,\bg^0) \Big\} \geq t -\epsilon
        \right)
    \\
        &\leq 
            \P\left(
                \sup_{(\bv,\bu) \in S} \Big\{ \frac2{\sqrt{n}}\< \bg_{\AMP}(\bv),\bv\> + f(\bv,\bu;\bR,\bS - \bR \bB,\bg^0) \Big\} \geq t
            \right)
    \\
        &\qquad\qquad\qquad\qquad\qquad\qquad\qquad\qquad\qquad+
            \P\left(
                \sup_{\| \bv \| \leq R}
                \Big|
                    \frac2{\sqrt{n}} \< \bg_{\AMP}(\bv) , \bv \>
                    -
                    \frac2{\sqrt{n}} \< \bg_{\SF} (\bv) , \bv \>
                \Big|
                > \epsilon
            \right).
    \end{aligned}
    \end{equation}
    Now we marginalize over $\bW,\bg^0$ and
    take $n \rightarrow \infty$,
    whence the second term on the right-hand side vanishes.
    Taking then $\epsilon \rightarrow 0$ gives Theorem \ref{thm:SF-post-AMP-gen}.
\end{proof}

\begin{proof}[Proof of Corollary \ref{cor:SF-post-AMP-Z2}]
    We can prove Corollary \ref{cor:SF-post-AMP-Z2} viewing $\bx$ as fixed, and then marginalizing over $\bx$.
    Thus, we now consider $\bx$ as fixed, and it does not show up explicitly in the proof.

    Corollary \ref{cor:SF-post-AMP-Z2} holds by exactly the same argument in the proof of Theorem \ref{thm:SF-post-AMP-gen},
    except with choices of $\bS_{\AMP}$, $\bR_{\AMP}$, and $\bB$ tailored to the iteration Eq.~\eqref{eq:AMP-Z2}.
    In particular, 
    now we choose 
    \begin{equation}
    \bB
        := 
        \begin{pmatrix}
            0      & \sbhat_1 & 0     & \cdots & 0 \\[5pt]
            0      & 0     & \sbhat_2 & & 0 \\[5pt]
            \vdots &       & \ddots     & \ddots & \vdots \\[5pt]
            0 &  0     &       & 0 & \sbhat_{k-1} \\[5pt]
            0      & 0     &    \cdots   & 0 & 0
        \end{pmatrix},
    \end{equation}
    and $\bS_{\AMP}(\bW,\bg^0) = \lambda \bM$ and $\bR_{\AMP}(\bW,\bg^0) = \bG + \lambda \bM \bB$,
    where $\bM$ and $\bG$ are as defined in Proposition \ref{prop:Z2-se}.
    With these choices and provided we interpret $\bK_{\leq k}$ as that state evolution matrix appearing in Proposition \ref{prop:Z2-se}, 
    the remainder of the proof goes through verbatim, except that Eq.~\eqref{eq:R-S-conv} 
    requires a slightly different justification.
    Now we have that $\bB$ is a random matrix,
    but by Proposition \ref{prop:Z2-se},
    it converges to the deterministic limit
    \begin{equation}
        \bB \pconv \tilde \bB := 
            \begin{pmatrix}
            0      & \sb_1 & 0     & \cdots & 0 \\[5pt]
            0      & 0     & \sb_2 & & 0 \\[5pt]
            \vdots &       & \ddots     & \ddots & \vdots \\[5pt]
            0 &  0     &       & 0 & \sb_{k-1} \\[5pt]
            0      & 0     &    \cdots   & 0 & 0
        \end{pmatrix},
    \end{equation}
    where $\sb_s = \lambda(1 - \E[\tanh^2(\gamma_s + \sqrt{\gamma_s}G)]$, where $G \sim \normal(0,1)$.
    Thus,
    \begin{equation}
        \frac1n \bR_{\AMP}^\top \bR_{\AMP} 
            =
            \frac1n \lambda^2 \bM^\top \bM \pconv \bK_{\leq k},
        \qquad 
        \frac1n \bR_{\AMP}^\top \bS_{\AMP}
            =
            \frac1n \lambda \bM^\top (\bG + \lambda \bM \bB)
            \pconv \tilde \bB^\top \bK_{\leq k} + \bK_{\leq k} \tilde \bB,
    \end{equation}
    which implies Eq.~\eqref{eq:R-S-conv}. 
    The remainder of the proof proceeds verbatim as before.
\end{proof}

\section{Auxiliary proofs}

\subsection{Proof of Proposition \ref{prop:Z2-se}}

Proposition \ref{prop:Z2-se} is equivalent to the state evolution results given in \cite{deshpande2016,celentanoFanMei2021,AlaouiMontanariSelke2022}.
Because we use slightly different notation,
we provide some details here for completeness.
As in Section \ref{sec:proofs},
by symmetry we may assume without loss of generality that $\bx = \ones$ and $X = 1$.

That $\gamma_s > 0$ for all $s$ holds by induction.
First, $\gamma_0 > 0$ by assumption. Then, if $ \gamma_s > 0$,
also $\gamma_{s+1} > 0$, because $\tanh^2(\gamma_s + \sqrt{\gamma_s}G)$ is positive almost surely.
Further the equation 
\begin{equation}
    \gamma = \lambda^2 \E[\tanh^2(\gamma + \sqrt{\gamma}G)] + \chi_{\sx} \gamma_0,        
\end{equation}
has a unique solution positive solution $\gamma_\infty > 0$:
for $\sx = \FMM$ and $\lambda > 1$, this is given by \cite[Proposition A.2]{celentanoFanMei2021}, and for $\sx = \AMS$ and $\lambda > 0$, this is given by \cite[Lemma 4.5]{AlaouiMontanariSelke2022} (note their $\gamma^*$ is our $\gamma_\infty - \gamma_0$ and their $t$ is our $\gamma_0$).
Moreover, by \cite[Lemma 6.1(a)]{deshpande2016} or \cite[Lemma 4.5(a)]{AlaouiMontanariSelke2022},
$\gamma \mapsto \lambda^2 \E[\tanh^2(\gamma + \sqrt{\gamma}G)]$ is strictly increasing and concave in $\gamma$.
Thus, $\gamma_s$ is strictly increasing and $\gamma_s \rightarrow \gamma_\infty$, provided $\gamma_0 < \gamma_\infty$. 
For $\sx = \AMS$, 
this holds by Eq.~\eqref{eq:def-gamp}, because $\chi_{\AMS} = 1$.
For $\sx = \FMM$, 
this holds by assumption.
Because $\gamma_s$ is strictly increasing, $\bGamma_{\leq k}$ is positive definite.
Thus, we have established items (i) and (iv).

Item (v) in the case $\sx = \FMM$ and $\lambda > 1$ is given by \cite[Proposition A.2]{celentanoFanMei2021}.
The case $\sx = \AMS$ is given by essentially the same argument appearing in \cite{celentanoFanMei2021},
which we provide here.
Indeed, consider a scalar observation $\gamma_\infty X + \sqrt{\gamma_\infty} G$, with $X \sim \Unif\{-1,+1\}$ and $G \sim \normal(0,1)$ independent. 
Then $\tanh(Y) = \E[X|Y]$, 
and its Bayes risk is given by $\E[(\tanh(Y) - X)^2] = 1 - \E[\tanh^2(Y)] = 1 - q_\infty$.
We can compare this to the risk of the linear estimator $(1 + \gamma_\infty)^{-1} Y$,
whose risk is easily computed to be $1/ (1+\gamma_\infty) \leq 1 / (1 + \lambda^2 q_\infty)$,
where the inequality holds by Eq.~\eqref{eq:def-gam-infty}.
Because the $\tanh(Y)$ and $(1 + \gamma_\infty)^{-1}Y$ differ on a set of positive probability,
we get $1 - q_\infty < 1 / (1 + \lambda^2q_\infty)$,
which rearranges to the first line in the display in item (v).
The second two lines follow from \cite[Appendix B.2]{deshpande2016},
also cited in \cite[Proposition A.2]{celentanoFanMei2021}.

Item (ii) follows from the state evolution result of \cite{bayatiMontanari2010},
together with some simplifications specific to the $\Z_2$-synchronization model,
as in \cite[Eqs. (69),(70)]{deshpande2016} and \cite[Proposition 4.4]{AlaouiMontanariSelke2022}.
Indeed, with $\bK_{\leq k}$ as defined in the lemma and using Eq.~\eqref{eq:AMP-Z2-se},
we have
\begin{equation}
\begin{aligned}
    K_{s+1,s+1}
        &= 
        \lambda^2 \E[\tanh(K_{s,s} + \chi_{\sx}\gamma_0 + \sqrt{K_{s,s} + \chi_{\sx}\gamma_0} G )].
\end{aligned}
\end{equation}
Moreover, as in the proof of \cite[Proposition 4.4]{AlaouiMontanariSelke2022},
\begin{equation}
\begin{aligned}
    &\E[\tanh(K_{s,s} + \chi_{\sx}\gamma_0 + \sqrt{K_{s,s} + \chi_{\sx}\gamma_0} G )\tanh(K_{t,t} + \chi_{\sx}\gamma_0 + \sqrt{K_{t,t} + \chi_{\sx}\gamma_0} G )]
    \\
    &\qquad\qquad =
    \E[\tanh(K_{s\wedge t,s\wedge t} + \chi_{\sx}\gamma_0 + \sqrt{K_{s\wedge t,s\wedge t} + \chi_{\sx}\gamma_0} G )] = K_{s+1,t+1}.
\end{aligned}
\end{equation}
Thus, $\bK_{\leq k}$,
as we have defined it,
satisfies the state evolution recursions as given in \cite{bayatiMontanari2010}.
(For another instance of this these type of state evolution simplifications, 
see Eqs.~(4.16) and (4.17) and Proposition 4.4 of \cite{AlaouiMontanariSelke2022}, which should be compared to the previous two displays).
Item (ii) now follows.
Moreover, note that the above two displays are equivalent to Item (iii).
The proof is complete.

\subsection{Proof of Lemma \ref{lem:det-ub}}
\label{sec:proof-lem-det-ub}

We first introduce the following notation:
for any $\mu \in W_2(\reals \times (-1,1)^{k+1} \times \reals^{k+2})$, 
let $\mu_1$ denote the marginal of its first two coordinates, and $\mu_2$ denotes the marginal of its last $2k+2$ coordinates.
Thus, if $(V,U,\brevebM,\brevebG,G_0,\Xi) \sim \mu$, 
then $(V,U) \sim \mu_1$ and $(\brevebM,\brevebG,\Xi) \sim \mu_2$.
Moreover, throughout the proof, $\epsilon(\cdot)$ will denote a $\reals_{>0} \cup \{ \infty \}$ valued function, possibly depending on $\lambda$ and $\bK_{\leq k}$,
with the property that $\epsilon(\delta) \rightarrow 0$ as $\delta \downarrow 0$.
We will also denote the space $W_2(\reals \times (-1,1)^{k+1} \times \reals^{k+1})$ with the short-hand $W_2$.
Moreover, denote $\muhat := \muhat_{\sqrt{n}\bv,\bu,\bM,\bG,\bg^0,\bxi}$,
so that $\muhat_1 = \muhat_{\sqrt{n}\bv,\bu}$ and $\muhat_2 = \muhat_{\bM,\bG,\bg^0,\bxi}$.

The continuity property we need is given by the following lemma.

\begin{lemma}\label{lem:cL-cont}
    Consider $\| \bv \| \leq 1$ and $ \bu \in [-1,1]^n$.
    Then there exists $\mu \in W_2(\reals \times (-1,1)^{k+1} \times \reals^{k+1})$ 
    such that $\mu_1 = \muhat_1$, $\mu_2 = \SE_{\sx,k}$, and
    \begin{equation}
        \big|
            \sF^{(1)}_{\sx,k}(\muhat) - \sF^{(1)}_{\sx,k}(\mu)
        \big|
            \leq 
            \epsilon\big( W_2(\muhat_2,\SE_{\sx,k}) \big),
    \end{equation}
    Moreover, $\mu$ can be chose so that additionally, for $(U^{(1)},V^{(1)},\brevebM^{(1)},\brevebG^{(1)},G_0^{(1)},\Xi^{(1)}) \sim \muhat$ and
    $(U^{(2)},V^{(2)},\allowbreak\brevebM^{(2)},\allowbreak\brevebG^{(2)},\allowbreak G_0^{(2)},\Xi^{(2)}) \sim \mu$,
    \begin{equation}
        \big|
            \big\| U^{(1)} - M_{k-1}^{(1)}\big\|_{L_2} 
            -
            \big\| U^{(2)} - M_{k-1}^{(2)}\big\|_{L_2} 
        \big| \leq W_2(\muhat_2,\SE_{\sx,k}).
    \end{equation}
\end{lemma}

\begin{proof}[Proof of Lemma \ref{lem:cL-cont}]
    Throughout the proof, $\epsilon(\,\cdot\,)$ is a function with the properties described in the lemma, but which may change at each appearance.

    Let $\Pi$ be the optimal coupling of $\muhat_2$ and $\SE_{\sx,k}$ (whose existence is guaranteed by \cite[Theorem 4.1]{villani2008optimal}).
    By the gluing lemma (see, e.g., \cite[pg.~23]{villani2008optimal}),
    there exist random variables $U,V$, $\brevebM^{(1)},\brevebG^{(1)},G_0^{(1)},\Xi^{(1)}$, $\brevebM^{(2)},\brevebG^{(2)},G_0^{(2)},\Xi^{(2)}$ on a shared probability space such that 
    $(U,V,\brevebM^{(1)},\brevebG^{(1)},G_0^{(1)},\Xi^{(1)}) \sim \muhat$
    and  
    $(\brevebM^{(1)},\brevebG^{(1)},\allowbreak G_0^{(1)},\allowbreak \Xi^{(1)},\allowbreak\brevebM^{(2)},\allowbreak\brevebG^{(2)},G_0^{(2)},\Xi^{(2)}) \sim \Pi$.
    We will show that the result holds for $\mu$ taken equal to the joint distribution of $(U,V,\allowbreak\brevebM^{(2)},\allowbreak\brevebG^{(2)},G_0^{(2)},\Xi^{(2)})$ under this coupling.

    First observe that the second display of the lemma follows by taking $U^{(1)} = U^{(2)} = U$ and using that $\big\| M_{k-1}^{(1)} - M_{k-2}^{(2)} \big\|_{L_2} \leq  W_2(\muhat_2,\SE_{\sx,k})$ by the optimality of the coupling $\Pi$.
    To derive the first display,
    observe that
    we may compute $\sF^{(1)}_{\sx,k}(\muhat)$ by plugging $U,V,\brevebM^{(1)},\brevebG^{(1)},G_0^{(1)},\Xi^{(1)}$ into Eq.~\eqref{eq:Leps-dist},
    and we may compute $\sF^{(1)}_{\sx,k}(\mu)$ by plugging $U,V,\brevebM^{(2)},\brevebG^{(2)},G_0^{(2)},\Xi^{(2)}$ into Eq.~\eqref{eq:Leps-dist},
    whence
    \begin{equation}
    \begin{aligned}
        \big|
                \sF^{(1)}_{\sx,k}(\muhat) - \sF^{(1)}_{\sx,k}(\mu)
            \big|
            &=
            2\lambda \big| \< G_{\AMP}^{(1)} , V \>_{L_2} - \< G_{\AMP}^{(2)} , V \>_{L_2} \big|
            \leq 
            2\lambda \big\| G_{\AMP}^{(1)} - G_{\AMP}^{(2)} \big\|_{L_2},
    \end{aligned}
    \end{equation}
    where $G_{\AMP}^{(1)}$ and $G_{\AMP}^{(2)}$ are defined as in Eq.~\eqref{eq:G-AMP},
    with respect to $V,\brevebM^{(1)},\brevebG^{(1)},\Xi^{(1)}$ and $V,\brevebM^{(2)},\brevebG^{(2)},\Xi^{(2)}$, respectively.

    Note that $\lambda^2 \E\big[\brevebM^{(2)}\brevebM^{(2)\top}\big]^{-1} = \bK_{\leq k}^{-1}$,
    and because $\bK_{\leq k} \succ 0$, 
    $\big\| \E\big[\brevebM^{(1)}\brevebM^{(1)\top}\big]^{-1} - \E\big[\brevebM^{(2)}\brevebM^{(2)\top}\big]^{-1} \big\|_{\op} \leq \eps( \| \brevebM^{(1)} - \brevebM^{(2)} \|_{L_2} ) \leq \epsilon\big(W_2(\muhat_2,\SE_{\sx,k})\big)$.
    Similarly,
    $\big\| \E\big[\brevebM^{(1)} V \big] - \E\big[\brevebM^{(2)}V\big] \big\| \leq \epsilon\big(W_2(\muhat_2,\SE_{\sx,k})\big)$.
    Because also $\| \brevebG^{(1)} - \brevebG^{(2)} \| \leq \epsilon\big(W_2(\muhat_2,\SE_{\sx,k})\big)$,
    repeated application of the Cauchy-Schwartz and triangle inequalities gives
    \begin{equation}
        \big\|
            \brevebG^{(1)\top} \E\big[\brevebM^{(1)}\brevebM^{(1)\top}\big]^\dagger \E\big[ \brevebM^{(1)} V \big]
            -
            \brevebG^{(2)\top} \E\big[\brevebM^{(2)}\brevebM^{(2)\top}\big]^\dagger \E\big[ \brevebM^{(2)} V \big]
        \big\|_{L_2}
        \leq \epsilon\big(W_2(\muhat_2,\SE_{\sx,k})\big).
    \end{equation}
    Representing, for $i \in \{1,2\}$,
    $\proj_{\brevebM^{(i)}}^\perp V = V - \brevebM^{(i)\top} \E\big[\brevebM^{(i)}\brevebM^{(i)\top}\big]^\dagger \E\big[\brevebM^{(i)}V\big] $,
    we can likewise conclude that $\big\| \proj_{\brevebM^{(1)}}^\perp V - \proj_{\brevebM^{(2)}}^\perp V \big\|_{L_2} \leq \epsilon\big(W_2(\muhat_2,\SE_{\sx,k})\big)$,
    whence in fact $\big| \big\| \proj_{\brevebM^{(1)}}^\perp V \|_{L_2} - \| \proj_{\brevebM^{(2)}}^\perp V \big\|_{L_2} \big| \leq \epsilon\big(W_2(\muhat_2,\SE_{\sx,k})\big)$.
    Because also $\| \Xi^{(1)} - \Xi^{(2)} \|_{L_2} \leq \epsilon\big(W_2(\muhat_2,\SE_{\sx,k})\big)$ and $\| \proj_{\brevebM^{(2)}}^\perp V \big\|_{L_2}  \leq 1$,
    we can conclude that 
    \begin{equation}
        \big\| 
            \| \proj_{\brevebM^{(1)}}^\perp V \|_{L_2} \Xi^{(1)}
            -
            \| \proj_{\brevebM^{(2)}}^\perp V \|_{L_2} \Xi^{(2)}
        \big\|_{L_2} \leq \epsilon\big(W_2(\muhat_2,\SE_{\sx,k})\big). 
    \end{equation}
    Combining with the previous displays yields the result.
\end{proof}

By Proposition \ref{prop:Z2-se}, 
as $n \rightarrow \infty$, 
\begin{equation}
    W_2\big(
        \muhat_{\bM,\bG,\bg^0,\bxi} ,\,
        \SE_{\sx,k}
    \big) 
    \pconv 0.
\end{equation}
Now consider $\bv \in \reals^n$, $\bu \in \reals^n$ with $\| \bv \| \leq 1$ and $\| \bu - \bbm^{k-1} \|/\sqrt{n} \leq \epsilon$,
and pick $\mu \in W_2(\reals \times (-1,1)^{k+1} \times \reals^{k+1})$ as in Lemma \ref{lem:cL-cont}.
Then,
for $(U,V,\brevebM,\brevebG,G_0,\Xi) \sim \mu$,
we have $\| V \|_{L_2} = \| \bv \| = 1$
and $\| U  - M_{k-1} \|_{L_2} \leq \| \bu - \bbm^{k-1} \| / \sqrt{n} + W_2(\muhat_2,\SE_{\sx,k}) \leq \epsilon +W_2(\muhat_2,\SE_{\sx,k}) $.
Thus,
$\mu \in \cS_{\sx,k}^{(1)}\big(\epsilon +W_2(\muhat_2,\SE_{\sx,k})\big)$.
Because also $\sF_{\sx,k}^{(1)}(\bv,\bu;\bM,\bG,\bxi) \leq \sF_{\sx,k}^{(1)}(\mu) + \epsilon(W_2(\muhat_2,\SE_{\sx,k}))$,
we conclude that 
\begin{equation}
    \sup_{\substack{ \| \bv \| = 1\\ 
                         \| \bu - \bbm^{k-1} \|/\sqrt{n} \leq \epsilon } }
            \sF_{\sx,k}(\bv,\bu;\bM,\bG,\bxi)
    \leq 
    \sup_{\mu \in \cS_{\sx,k}^{(1)}\big(\epsilon +W_2(\muhat_2,\SE_{\sx,k})\big)}
            \sF^{(1)}_{\sx,k}(\mu) + \epsilon(W_2(\muhat_2,\SE_{\sx,k})).
\end{equation} 
By Proposition \ref{prop:Z2-se}, $W_2(\muhat_2,\SE_{\sx,k}) \pconv =0$ as $n\rightarrow \infty$,
we get 
\begin{equation}
    \limsup_{n \rightarrow \infty}
        \sup_{\substack{ \| \bv \| = 1\\ 
                         \| \bu - \bbm^{k-1} \|/\sqrt{n} \leq \epsilon } }
            \sF_{\sx,k}(\bv,\bu;\bM,\bG,\bxi)
    \leq 
    \sup_{\mu \in \cS_{\sx,k}^{(1)}(2\epsilon)}
            \sF^{(1)}_{\sx,k}(\mu).
\end{equation}
Now taking $k \rightarrow \infty$ followed by $\epsilon \rightarrow 0$,
the result of the lemma follows.

\subsection{Proof of Theorem \ref{thm:local-strong-cvx}(ii)}
\label{app:unique-stat}

Note that for $\sx = \FMM$,
\begin{equation}
\begin{aligned}
    \nabla \cF_{\sx}(\bbm^{k-1})
        &=
        -\frac{\lambda}{n}\bY \bbm^{k-1} + \frac1n \tanh^{-1}(\bbm^{k-1})
        +\frac{\lambda^2}{n}(1 - Q(\bbm^{k-1}))\bbm^{k-1}
    \\
        &=
        -\frac1n(\bz^k - \bz^{k-1})
        + \frac{\lambda^2}{n}(1-Q(\bbm^{k-1}))(\bbm^{k-1} - \bbm^{k-2}).
\end{aligned}
\end{equation}
By Theorem \ref{thm:local-strong-cvx}(i),
there exists $c,c' > 0$ such that for any $\epsilon < c'$, there exists large enough $k$ such that, with probability approaching $1$ as $n \rightarrow \infty$,
the TAP free energy is $c/n$ strongly convex on $\{ \| \bu - \bbm^{k-1} \| / \sqrt{n} \leq \epsilon \} \cup [-1,1]^n$.
Thus, provided $\| \nabla \cF_{\sx}(\bbm^{k-1}) \| \leq c\epsilon/(4\sqrt{n}) $,
we get that $\cF_{\sx}$ has a unique stationary point $\bbm^*$ on $\{ \| \bu - \bbm^{k-1} \| / \sqrt{n} \leq \epsilon \} \cup [-1,1]^n$,
with $\| \bbm^* - \bbm^{k-1} \|/\sqrt{n} \leq (2n/c) \| \nabla \cF_{\sx}(\bbm^{k-1}) \| / \sqrt{n} \leq \epsilon / 2$.
Using that $\bz^s = \bg^s + \lambda (\< \bx , \bbm^{s-1} \> /n) \bbm^{s-1}$,
we have by Proposition \ref{prop:Z2-se} that
\begin{equation}
\begin{aligned}
    \plimsup_{n \rightarrow \infty} \frac1{\sqrt{n}} \| \bz^k - \bz^{k-1} \|
        &\leq
        \plim_{n\rightarrow \infty}
        \Big\{
            \frac1{\sqrt{n}} \| \bg^k - \bg^{k-1} \| 
            + 
            \frac{\lambda^2\|\bx\|}{\sqrt{n}} 
            \Big|
                \frac{\< \bx , \bbm^{k-1}\>}n
                -
                \frac{\< \bx , \bbm^{k-2}\>}n
            \Big|
        \Big\}
    \\
        &= 
        \sqrt{K_{k,k} + K_{k-1,k-1} - 2 K_{k,k-1}}
        + 
        |\gamma_k - \gamma_{k-1}|
        =
        \sqrt{\gamma_k - \gamma_{k-1}}
        +
        |\gamma_k - \gamma_{k-1}|.
\end{aligned}
\end{equation}
As $k \rightarrow \infty$, the right-hand side converges to 0.
Similarly, letting $q_k = \E[\tanh^2(\gamma_{k-1} + \sqrt{\gamma_{k-1}}G)] = (\gamma_k - \chi_{\sx}\gamma_0)/\lambda^2$,
we have
\begin{equation}
    \plim_{n \rightarrow \infty}(1-Q(\bbm^{k-1}))\frac{\|\bbm^{k-1} - \bbm^{k-2} \|}{\sqrt{n}}
        = 
        (1-q_k) \sqrt{q_k - q_{k-1}},
\end{equation}
which also goes to 0 as $k \rightarrow \infty$.
Thus, for sufficiently large $k$ depending on $\lambda,\gamma_0,c,\epsilon$,
we have $\| \nabla \cF_{\sx}(\bbm^{k-1})\| < c \epsilon/(4\sqrt{n})$ with probability going to 1 as $n \rightarrow \infty$. 
Thus, we confirmed that for sufficiently large $k$,
$\cF_{\sx}$ has a unique stationary point on $\{ \| \bu - \bbm^{k-1} \| / \sqrt{n} \leq \epsilon \} \cup [-1,1]^n$ with high probability.
To see that this stationary point must lie also in $(-1,1)^n$,
note that as $\bbm$ converges to the boundary of $[-1,1]^n$,
we have $\|\nabla \cF_{\sx}(\bbm)\|$  diverges to $\infty$, because $\| \tanh^{-1}(\bbm)\| $ diverges to $\infty$ and the other terms in $\nabla \cF_{\sx}(\bbm)$ remain bounded.

The argument for $\sx = \AMS$ is nearly equivalent. 
Now we have 
\begin{equation}
\begin{aligned}
    \nabla \cF_{\sx}(\bbm^{k-1})
        &=
        -\frac{\lambda}{n}\bY \bbm^{k-1} + \frac1n \tanh^{-1}(\bbm^{k-1})
        + \frac1n\by
        +\frac{\lambda^2}{n}(1 - q_\infty)\bbm^{k-1}
    \\
        &=
        -\frac1n(\bz^k - \bz^{k-1})
        + \frac{\lambda^2}{n}(Q(\bbm^{k-1}) - q_\infty) \bbm^{k-1}
        + \frac{\lambda^2}{n}(1-Q(\bbm^{k-1}))(\bbm^{k-1}-\bbm^{k-2}).
\end{aligned}
\end{equation}
The argument then proceeds exactly as before, 
except we must now deal also with the term $(\lambda^2/n)(Q(\bbm^{k-1})-q_\infty)\bbm^{k-1}$,
which we must also show satisfies
\begin{equation}
    \lim_{k \rightarrow \infty} \plim_{n \rightarrow \infty} (\lambda^2/n)(Q(\bbm^{k-1})-q_\infty)\|\bbm^{k-1}\| = 0.
\end{equation}
Because $\plim_{n \rightarrow \infty} Q(\bbm^{k-1}) = q_k$, 
and $q_k \rightarrow q_\infty$ as $k \rightarrow \infty$,
the previous display holds, and the result follows.

\end{document}